\documentclass[a4paper,12pt]{article}
\usepackage{tikz}
\usetikzlibrary{matrix,arrows,decorations.markings}
\usepackage{amsmath,amsfonts,amssymb,amsthm,url,textcomp}
\usepackage{csquotes}
\usepackage{a4wide}
\usepackage[numbers]{natbib}
\usepackage{fancyhdr}
\usepackage{anyfontsize}
\usepackage{hyperref}
\usepackage{hhline}
\usepackage{leftidx}
\usepackage{mdframed}
\usepackage{enumitem}

\allowdisplaybreaks

\newtheorem{theorem}{Theorem}\numberwithin{theorem}{section}
\newtheorem{definition}[theorem]{Definition}
\newtheorem{lemma}[theorem]{Lemma}
\newtheorem{corollary}[theorem]{Corollary}
\newtheorem{proposition}[theorem]{Proposition}
\newtheorem{notation}[theorem]{Notation}

\newtheorem{question}[theorem]{Question}
\newtheorem{problem}[theorem]{Problem}
\newtheorem{theoremm}{Theorem}\numberwithin{theoremm}{subsection}
\newtheorem{deffinition}[theoremm]{Definition}
\newtheorem{lemmma}[theoremm]{Lemma}

\newtheorem{propposition}[theoremm]{Proposition}

\numberwithin{theoremmm}{subsubsection}

\theoremstyle{remark}
\newtheorem{remark}[theorem]{Remark}

\newcommand{\Rad}{\operatorname{Rad}}
\newcommand{\Aut}{\operatorname{Aut}}
\newcommand{\Alt}{\operatorname{Alt}}
\newcommand{\PSL}{\operatorname{PSL}}

\newcommand{\ord}{\operatorname{ord}}
\newcommand{\Sym}{\operatorname{Sym}}

\newcommand{\C}{\operatorname{C}}

\newcommand{\Soc}{\operatorname{Soc}}
\newcommand{\Inn}{\operatorname{Inn}}
\newcommand{\id}{\operatorname{id}}
\newcommand{\N}{\operatorname{N}}

\newcommand{\e}{\mathrm{e}}

\newcommand{\GL}{\operatorname{GL}}

\newcommand{\Mod}[1]{\ (\textup{mod}\ #1)}

\newcommand{\F}{\operatorname{F}}

\newcommand{\IN}{\mathbb{N}}

\newcommand{\Sp}{\operatorname{Sp}}
\newcommand{\perm}{\mathrm{perm}}

\newcommand{\IF}{\mathbb{F}}

\newcommand{\End}{\operatorname{End}}

\newcommand{\IZ}{\mathbb{Z}}

\newcommand{\Hom}{\operatorname{Hom}}

\newcommand{\maol}{\operatorname{maol}}

\newcommand{\Exp}{\operatorname{Exp}}
\newcommand{\Dih}{\operatorname{Dih}}
\newcommand{\cent}{\mathrm{cent}}
\newcommand{\mccl}{\operatorname{mccl}}

\begin{document}

\title{Finite groups with only small automorphism orbits}

\author{Alexander Bors\thanks{Johann Radon Institute for Computational and Applied Mathematics (RICAM), Altenbergerstra{\ss}e 69, 4040 Linz, Austria. \newline E-mail: \href{mailto:alexander.bors@ricam.oeaw.ac.at}{alexander.bors@ricam.oeaw.ac.at} \newline The author is supported by the Austrian Science Fund (FWF), project J4072-N32 \enquote{Affine maps on finite groups}. \newline 2010 \emph{Mathematics Subject Classification}: Primary: 20D45. Secondary: 20D60, 20F24. \newline \emph{Key words and phrases:} Finite groups, Group actions, Automorphism orbits, Inhomogeneous structures}}

\date{\today}

\maketitle

\abstract{We study finite groups $G$ such that the maximum length of an orbit of the natural action of the automorphism group $\Aut(G)$ on $G$ is bounded from above by a constant. Our main results are the following: Firstly, a finite group $G$ only admits $\Aut(G)$-orbits of length at most $3$ if and only if $G$ is cyclic of one of the orders $1$, $2$, $3$, $4$ or $6$, or $G$ is the Klein four group or the symmetric group of degree $3$. Secondly, there are infinitely many finite ($2$-)groups $G$ such that the maximum length of an $\Aut(G)$-orbit on $G$ is $8$. Thirdly, the order of a $d$-generated finite group $G$ such that $G$ only admits $\Aut(G)$-orbits of length at most $c$ is explicitly bounded from above in terms of $c$ and $d$. Fourthly, a finite group $G$ such that all $\Aut(G)$-orbits on $G$ are of length at most $23$ is solvable.}

\section{Introduction}\label{sec1}

The study of \emph{structures} $X$ (in the model-theoretic sense, i.e., sets endowed with operations and relations) that are \enquote{highly symmetrical}, expressed through transitivity assumptions on natural actions of the automorphism group $\Aut(X)$, has a long and rich history, during which various strong theories have been built and beautiful results have been obtained. As examples, we mention vertex-transitive graphs \cite[Definition 4.2.2, p.~85]{BW79a}, block-transitive designs \cite{CP93a,CP93b}, and finite flag-transitive projective planes \cite{Tha03a}.

When $X$ is a group $G$, the assumption that $\Aut(G)$ acts transitively on $G$ is not interesting, as only the trivial group satisfies it. Therefore, weaker conditions have been proposed and studied, such as the following (assuming that $G$ is \emph{finite}):
\begin{enumerate}
\item \enquote{$\Aut(G)$ admits exactly $c$ orbits on $G$.} for some given, small constant $c$. For $c=2$, it is not difficult to show that this is equivalent to $G$ being nontrivial and elementary abelian. For results concerning $c\in\{3,4,5,6,7\}$, see the papers \cite{BD16a,DGB17a,LM86a,Stro02a} by various authors.
\item \enquote{$\Aut(G)$ admits at least one orbit of length at least $\rho|G|$ on $G$.} for some given constant $\rho\in\left(0,1\right]$. For example, it is known that if $\rho>\frac{18}{19}$, then $G$ is necessarily solvable \cite[Theorem 1.1.2(1)]{Bor19a}.
\item \enquote{For each element order $o$ in $G$, $\Aut(G)$ acts transitively on elements of order $o$ in $G$.}. In other words, $\Aut(G)$ is \enquote{as transitive as possible} in view of the fact that automorphisms must preserve the orders of elements. Such finite groups $G$ are called \emph{AT-groups} and are studied extensively by Zhang in \cite{Zha92a}.
\end{enumerate}

In this paper, we are not concerned with such \enquote{highly homogeneous} finite groups, but rather with finite groups $G$ that are \enquote{highly inhomogeneous} in the sense that they only admit small $\Aut(G)$-orbits (i.e., of constantly bounded length). There is also some relevant literature in this context, most notably the 1984 paper \cite{RW84a} by Robinson and Wiegold, in which they characterize general (not necessarily finite) such groups structurally \cite[Theorem 1]{RW84a} and provide, for each prime $p$, an example of an infinite $p$-group $G_p$ of nilpotency class $2$ and of exponent $p^2$ such that $\Aut(G_p)$ is uncountably infinite but only has orbits of length at most $p^2(p-1)^2$ on $G_p$ \cite[Proposition 3 and the Remark after its proof]{RW84a}. Another noteworthy result in this regard is that there are uncountable abelian groups with only two automorphisms, see e.g.~\cite[Theorem II]{Gro57a}.

However, finite groups behave quite differently to infinite groups in many regards, and by a result of Ledermann and B.H.~Neumann \cite{LN56a}, as the order of a finite group $G$ tends to $\infty$, so does the order of $\Aut(G)$. In other words: \enquote{Large finite groups have many automorphisms.} Based on this result, one might conjecture that even the following stronger assertion holds: \enquote{For finite groups $G$, as $|G|\to\infty$, the maximum length of an $\Aut(G)$-orbit on $G$ tends to $\infty$ as well}. This, however, is not true, see our Theorem \ref{mainTheo}(2).

Throughout the rest of this paper, we denote by $\maol(G)$ the maximum length of an $\Aut(G)$-orbit on the finite group $G$. Moreover, $\exp$ and $\log$ denote the natural exponential and logarithm function respectively (with base the Euler constant $\e$). We now state our main results:

\begin{theorem}\label{mainTheo}
The following statements hold:
\begin{enumerate}
\item For each finite group $G$, the following are equivalent:
\begin{enumerate}
\item $\maol(G)\leqslant3$.
\item $G$ is isomorphic to one of the following: $\IZ/m\IZ$ with $m\in\{1,2,3,4,6\}$, $(\IZ/2\IZ)^2$, $\Sym(3)$.
\end{enumerate}
In particular, there are only finitely many finite groups $G$ with $\maol(G)\leqslant3$.
\item There are infinitely many finite $2$-groups $G$ with $\maol(G)=8$.
\item For each pair $(c,d)$ of positive integers and every $d$-generated finite group $G$ with $\maol(G)\leqslant c$, we have that $\log{|G|}$ is at most
\[
1.01624d\cdot\left(A(c,d)+1\right)\cdot\left(\frac{\log{A(c,d)}}{\log{2}}+1\right)+\frac{1}{2}(7+\log{c})\log{c}
\]
where
\[
A(c,d):=c^{d+\frac{1}{2}(7+\log{c})\left({d \choose 2}+\frac{d}{2\log{2}}\cdot(7+\log{c})\log{c}\right)}.
\]
\item A finite group $G$ with $\maol(G)\leqslant23$ is solvable.
\end{enumerate}
\end{theorem}

We note that the constant $23$ in Theorem \ref{mainTheo}(4) is optimal, as $\maol(\Alt(5))=24$.

\section{Some preparations}\label{secPrep}

In this section, we list some notation that will be used throughout the paper, and we discuss a few basic facts concerning power-commutator presentations, central automorphisms and finite groups without nontrivial solvable normal subgroups.

\subsection{Notation}\label{subsecPrep1}

We denote by $\IN$ the set of natural numbers (including $0$) and by $\IN^+$ the set of positive integers. For a prime power $q$, the notation $\IF_q$ stands for the finite field with $q$ elements. The identity function on a set $X$ is denoted by $\id_X$. The Euler totient function will be denoted by $\phi$ throughout and is to be distinguished from the symbol $\varphi$ reserved for group homomorphisms. The kernel of a group homomorphism $\varphi$ is denoted by $\ker(\varphi)$, and the order of an element $g$ of a group $G$ by $\ord(g)$, sometimes also by $\ord_G(g)$ for greater clarity. When $g$ and $h$ are elements of a group $G$, then we denote by $[g,h]:=g^{-1}h^{-1}gh$ the commutator of $g$ and $h$, and for subsets $X,Y\subseteq G$, the notation $[X,Y]$ stands for the subgroup of $G$ generated by the commutators $[x,y]$ with $x\in X$ and $y\in Y$. We always denote the quotient of a group $G$ by a normal subgroup $N$ by $G/N$ and reserve the notation $X\setminus Y$ for the set-theoretic difference of the sets $X$ and $Y$. The index of a subgroup $H$ in a group $G$ is written $|G:H|$. If $x_1,\ldots,x_n$ are pairwise distinct variables, then $\F(x_1,\ldots,x_n)$ stands for the free group generated by $x_1,\ldots,x_n$. When $A$ is an abelian group, then the semidirect product $A\rtimes\IZ/2\IZ$, where the generator of $\IZ/2\IZ$ acts on $A$ by inversion, is called the \emph{generalized dihedral group over $A$} and will be denoted by $\Dih(A)$. The symmetric group on a set $X$ is denoted by $\Sym(X)$, and for $n\in\IN^+$, the symmetric and alternating group of degree $n$ are written $\Sym(n)$ and $\Alt(n)$ respectively. All group actions discussed in this paper are on the right, and when $\varphi: G\rightarrow\Sym(X)$ is an action of the group $G$ on the set $X$, then for $g\in G$ and $x\in X$, we write $x^g$ shorthand for $\varphi(g)(x)$, and we write $x^G$ for the full orbit of $x$ under $G$. The exponent (i.e., least common multiple of the element orders) of a finite group $G$ is denoted by $\Exp(G)$ (to be distinguished from the notation $\exp$ reserved for the natural exponential function), and the smallest size of a generating subset of $G$ is denoted by $d(G)$. The notation $\Rad(G)$ is used for the solvable radical (largest solvable normal subgroup) of a finite group $G$, and $\Soc(G)$ is used for the socle (product of all the minimal nontrivial normal subgroups) of $G$; see also Subsection \ref{subsecPrep3}. The center of a group $G$ is denoted by $\zeta G$, and $G':=[G,G]$ denotes the commutator subgroup of $G$. The inner automorphism group of a group $G$ is written $\Inn(G)$. If $G$ and $H$ are groups, then $\End(G)$ denotes the set (monoid) of endomorphisms of $G$, and $\Hom(G,H)$ denotes the set of group homomorphisms $G\rightarrow H$.

\subsection{Power-commutator presentations of finite solvable groups}\label{subsecPrep2}

A group $G$ is called \emph{polycyclic} if and only if it admits a \emph{polycyclic series}, that is, a subnormal series $G=G_1\unrhd G_2\unrhd\cdots\unrhd G_n\unrhd G_{n+1}=\{1_G\}$ such that all the factors $G_i/G_{i+1}$, with $i\in\{1,\ldots,n\}$, are cyclic. A generating tuple $(g_1,\ldots,g_n)$ of $G$ is called a \emph{polycyclic generating sequence of $G$} if and only if, setting $G_i:=\langle g_i,g_{i+1},\ldots,g_n\rangle$ for $i=1,\ldots,n+1$, the subgroup series $G=G_1\geqslant G_2\geqslant\cdots\geqslant G_n\geqslant G_{n+1}=\{1_G\}$ is a polycyclic series in $G$. Clearly, every polycyclic group is solvable, and all \emph{finite} solvable groups are polycyclic. If $G$ is a polycylic group and $(g_1,\ldots,g_n)$ is a polycyclic generating sequence of $G$, then with respect to the generating tuple $(g_1,\ldots,g_n)$, $G$ can be represented by a so-called \emph{polycyclic presentation}, see e.g.~\cite[Theorem 8.8, p.~279]{HEO05a}. For our purposes, it will be more convenient to work with a variant of polycyclic presentations called \emph{power-commutator presentations}. Assume that $G$ is a \emph{finite} polycylic group (the finiteness assumption is not essential, but makes the situation a bit simpler) and that $(g_1,\ldots,g_n)$ is a polycyclic generating sequence of $G$. Then with respect to the generating tuple $(g_1,\ldots,g_n)$, the group $G$ has a power-commutator presentation of the form
\begin{align*}
G=\langle x_1,\ldots,x_n \mid &x_i^{e_i}=x_{i+1}^{a_{i,i+1}}\cdots x_n^{a_{i,n}}\text{ for }i=1,\ldots,n; \\
&[x_i,x_j]=x_{i+1}^{b_{i,j,i+1}}\cdots x_n^{b_{i,j,n}}\text{ for }1\leqslant i<j\leqslant n\rangle,
\end{align*}
where for $i=1,\ldots,n$, the formal generator $x_i$ corresponds to the group element $g_i$, and $e_i$ is the so-called \emph{relative order} of $g_i$, i.e., the order of $g_iG_{i+1}=g_i\langle x_{i+1},\ldots,x_n\rangle$ in $G/G_{i+1}$. Moreover, the exponents $a_{i,k}$ for $i=1,\ldots,n$ and $k=i+1,\ldots,n$, and the exponents $b_{i,j,k}$ for $1\leqslant i<j\leqslant n$ and $k=i+1,\ldots,n$ are integers in $\{0,1,\ldots,e_k-1\}$. For more details on polycylic groups, see \cite[Chapter 8]{HEO05a}.

\subsection{Central automorphisms}\label{subsecPrep3}

If $G$ is a group and $f$ is a group homomorphism $G\rightarrow \zeta G$, then it is easy to check that the function $\varphi_f:G\rightarrow G$, $g\mapsto gf(g)$, is a group endomorphism of $G$, and that conversely, every endomorphism of $G$ which leaves each coset of $\zeta G$ in $G$ set-wise invariant is of this form. Such endomorphisms of $G$ are called \emph{central}. Moreover, a central endomorphism $\varphi_f$ of a group $G$ has trivial kernel if and only the neutral element $1_G$ is the only element of $\zeta G$ which is mapped to its own inverse by $f$. In the case of finite groups $G$, the central endomorphisms of $G$ with trivial kernel are the \emph{central automorphisms} of $G$, which form a subgroup of $\Aut(G)$ denoted by $\Aut_{\cent}(G)$.

\subsection{Finite semisimple groups}\label{subsecPrep4}

Throughout this paper, the term \enquote{semisimple group} denotes a group without nontrivial solvable normal subgroups; for finite groups $G$, this is equivalent to the condition that the solvable radical $\Rad(G)$ is trivial. Note that since the class of solvable groups is closed under group extensions, for every finite group $G$, the quotient $G/\Rad(G)$ is semisimple. Moreover, for finite semisimple groups $H$, the structure of $H$ is controlled by the socle $\Soc(H)$. More precisely, $\Soc(H)$ is a direct product of nonabelian finite simple groups, and $H$ acts faithfully on $\Soc(H)$ via conjugation, so that, up to isomorphism, $H$ may be viewed as a subgroup of $\Aut(\Soc(H))$ containing $\Inn(\Soc(H))$; see also \cite[result 3.3.18, p.~89]{Rob96a}.

\section{Finite groups \texorpdfstring{$G$}{G} with \texorpdfstring{$\maol(G)\leqslant3$}{maol(G)<=3}}\label{sec2}

This section is concerned with the proof of Theorem \ref{mainTheo}(1). We will go through the three cases $\maol(G)=1,2,3$ separately, but first, we prove the following simple lemma, which will be used frequently:

\begin{lemma}\label{simpleLem}
The following hold:
\begin{enumerate}
\item Let $G_1,\ldots,G_k$ be finite groups. Then
\[
\maol(\prod_{k=1}^n{G_k})\geqslant\prod_{k=1}^n{\maol(G_k)}\geqslant\max\{\maol(G_k)\mid k=1,\ldots,n\}.
\]
\item For every finite abelian group $G$, we have $\maol(G)\geqslant\phi(\Exp(G))$, where $\phi$ denotes the Euler totient function.
\item Let $G$ be a finite nilpotent group. Then $\maol(G)=\prod_p{\maol(G_p)}$ where the index $p$ ranges over the primes and $G_p$ denotes the (unique) Sylow $p$-subgroup of $G$.
\end{enumerate}
\end{lemma}

\begin{proof}
For statement (1): This holds since $\prod_{k=1}^n{\Aut(G_k)}$ embeds into $\Aut(\prod_{k=1}^n{G_k})$ via \enquote{component-wise mapping}.

For statement (2): First, note that if $G$ is cyclic, then $\maol(G)=\phi(\Exp(G))$, as $\phi(\Exp(G))=\phi(|G|)$ is just the number of generators of $G$. If $G$ is a general finite abelian group, then by the structure theorem for finite abelian groups, $G$ has a cyclic direct factor of order $\Exp(G)$, and the asserted inequality follows by statement (1).

For statement (3): This is clear since $\Aut(G)$ is isomorphic to the direct product $\prod_p{\Aut(G_p)}$ via \enquote{component-wise mapping}.
\end{proof}

\subsection{Finite groups \texorpdfstring{$G$}{G} with \texorpdfstring{$\maol(G)=1$}{maol(G)=1}}\label{subsec2P1}

The following proposition, whose proof is given for completeness, is easy and well-known:

\begin{propposition}\label{maol1Prop}
Let $G$ be a finite group. The following are equivalent:
\begin{enumerate}
\item $\maol(G)=1$.
\item $\Aut(G)$ is trivial.
\item $G\cong\IZ/m\IZ$ with $m\in\{1,2\}$.
\end{enumerate}
\end{propposition}

\begin{proof}
For \enquote{(1) $\Rightarrow$ (2)}: Assume that $\maol(G)=1$, and let $\alpha\in\Aut(G)$. Then for each $g\in G$, we have $g^{\alpha}\in g^{\Aut(G)}=\{g\}$, so that $g^{\alpha}=g$ and $\alpha=\id_G$. Since $\alpha\in\Aut(G)$ was arbitrary, it follows that $\Aut(G)=\{\id_G\}$, as required.

For \enquote{(2) $\Rightarrow$ (3)}: Assume that $\Aut(G)$ is trivial. Since $G/\zeta G \cong \Inn(G)\leqslant\Aut(G)$, it follows that $G=\zeta G$, i.e., $G$ is abelian. Writing $G$ additively, we find that the inversion on $G$, $-\id_G$, is an automorphism of $G$, and so $-\id_G=\id_G$, i.e., $G$ is of exponent $2$, and thus $G\cong(\IZ/2\IZ)^d$ for some $d\in\IN$. But if $d\geqslant2$, then by Lemma \ref{simpleLem}(1),
\[
\maol(G)\geqslant\maol((\IZ/2\IZ)^2)=3>1,
\]
so that $\Aut(G)$ must be nontrivial, a contradiction. Hence $d\in\{0,1\}$, as required.

For \enquote{(3) $\Rightarrow$ (1)}: Assume that $G$ is of order at most $2$. Then since $\Aut(G)$ is contained in a point stabilizer in $\Sym(G)$, every element of $G$ must be fixed by all permutations in $\Aut(G)$, whence $\maol(G)=1$, as required.
\end{proof}

\subsection{Finite groups \texorpdfstring{$G$}{G} with \texorpdfstring{$\maol(G)=2$}{maol(G)=2}}\label{subsec2P2}

These groups are less trivial to deal with than the ones with $\maol$-value $1$. Note that if $G$ is a finite group with $\maol(G)=2$, then $\alpha^2=\id_G$ for every automorphism $\alpha$ of $G$. Hence $\Aut(G)$ is of exponent $2$, i.e., $\Aut(G)$ is an elementary abelian $2$-group. We will need a few results on finite groups with abelian automorphism group.

\begin{deffinition}\label{millerGroupDef}
A nonabelian finite group $G$ with abelian automorphism group is called a \emph{Miller group}.
\end{deffinition}

This terminology, taken from the survey paper \cite{KY18a}, is in honor of G.A.~Miller, who gave the first example of such a group (of order $64$) in 1913 \cite{Mil13a} (see also \cite[Section 3, (3.1)]{KY18a}). Since then, a rich theory of Miller groups with many beautiful results and examples has emerged. We will need the following:

\begin{propposition}\label{millerGroupProp}
Let $G$ be a Miller group. Then the following hold:
\begin{enumerate}
\item $G$ is nilpotent of class $2$.
\item Every Sylow subgroup of $G$ has abelian automorphism group.
\item If $G$ is a $p$-group for some prime $p$ and $|G'|>2$, then $G'$ is not cyclic.
\end{enumerate}
\end{propposition}

\begin{proof}
For statement (1): This holds since for every group $H$, being nilpotent of class at most $2$ is equivalent to the commutativity of $\Inn(H)$; see also \cite[Section 1]{KY18a}.

For statement (2): This is clear since $\Aut(G)$ is the direct product of the automorphism groups of the Sylow subgroups of $G$ (see also the proof of Lemma \ref{simpleLem}(3)).

For statement (3): By \cite[statement (4) at the end of Section 1]{KY18a}, this holds if one additionally assumes that $G$ is \emph{purely nonabelian}, i.e., $G$ has no nontrivial abelian direct factor. However, this additional assumption can be dropped, for if $G$ is not purely nonabelian, then $G=G_0\times A$ where $G_0$ is purely nonabelian and $A$ is abelian. Since $\Aut(G_0)$ embeds into $\Aut(G)$, we have that $G_0$ is also a Miller $p$-group, and $|G'|=|G_0'|>2$, so $G_0'$ is not cyclic. But $G'\cong G_0'$, whence $G'$ is not cyclic.
\end{proof}

We can now prove the following lemma, which will be used in our proof of the classification of finite groups $G$ with $\maol(G)=2$ (see Proposition \ref{maol2Prop} below).

\begin{lemmma}\label{maol2Lem}
Let $G$ be a finite group with $\maol(G)=2$. Then the following hold:
\begin{enumerate}
\item If $G$ is abelian, then $G\cong\IZ/m\IZ$ with $m\in\{3,4,6\}$.
\item If $G$ is nonabelian, then
\begin{enumerate}
\item $G$ is a Miller $2$-group.
\item $\zeta G$ is cyclic.
\item $|G'|=2$.
\item $|\zeta G|>2$.
\item $G/G'$ is an elementary abelian $2$-group.
\end{enumerate}
\end{enumerate}
\end{lemmma}

\begin{proof}
For statement (1): By Lemma \ref{simpleLem}(3), $\maol(G)=\prod_p{\maol(G_p)}$, where the index $p$ ranges over the primes and $G_p$ denotes the Sylow $p$-group of $G$. Moreover, if $G_p$ is nontrivial, then by Lemma \ref{simpleLem}(2), $\maol(G_p)\geqslant\phi(p)=p-1$. It follows that $G_p$ is trivial unless $p\in\{2,3\}$, i.e., $G$ is a finite abelian $\{2,3\}$-group. Consider the following cases:
\begin{enumerate}
\item Case: $G$ is a $2$-group. By Lemma \ref{simpleLem}(2), $\maol(G)\geqslant\phi(\Exp(G))$, and thus $\Exp(G)\leqslant4$. But $\Exp(G)=2$ is impossible, since then $G\cong(\IZ/2\IZ)^d$ for some $d\geqslant 2$, and thus $\maol(G)\geqslant\maol((\IZ/2\IZ)^2)=3$ by Lemma \ref{simpleLem}(1). Hence $\Exp(G)=4$. If $G$ has more than one direct factor $\IZ/4\IZ$ in its decomposition into primary cyclic groups, then by Lemma \ref{simpleLem}(1), $\maol(G)\geqslant\maol((\IZ/4\IZ)^2)=12$, a contradiction. Hence $G\cong(\IZ/2\IZ)^d\times\IZ/4\IZ$ for some $d\in\IN$. If $d\geqslant1$, then by Lemma \ref{simpleLem}(1), $\maol(G)\geqslant\maol(\IZ/2\IZ\times\IZ/4\IZ)=4$, a contradiction. It follows that $G\cong\IZ/4\IZ$.
\item Case: $G$ is a $3$-group. Again, we have $\maol(G)\geqslant\phi(\Exp(G))$ by Lemma \ref{simpleLem}(2), which implies that $\Exp(G)=3$, whence $G\cong(\IZ/3\IZ)^d$ for some $d\in\IN^+$. If $d\geqslant2$, then by Lemma \ref{simpleLem}(1), $\maol(G)\geqslant\maol((\IZ/3\IZ)^2)=8$, a contradiction. Hence $G\cong\IZ/3\IZ$.
\item Case: $G$ is neither a $2$- nor a $3$-group. Then $G=G_2\times G_3$ with $G_p$ a nontrivial abelian $p$-group for $p\in\{2,3\}$. By Lemma \ref{simpleLem}(3), we have
\[
2=\maol(G)=\maol(G_2)\cdot\maol(G_3).
\]
Hence $(\maol(G_2),\maol(G_3))$ is either $(2,1)$ or $(1,2)$. But the former is impossible, since by Proposition \ref{maol1Prop}, there are no nontrivial finite $3$-groups with $\maol$-value $1$. Hence $\maol(G_2)=1$ and $\maol(G_3)=2$. It follows by Proposition \ref{maol1Prop} and the previous Case that $G_2\cong\IZ/2\IZ$ and $G_3\cong\IZ/3\IZ$, whence $G\cong\IZ/6\IZ$.
\end{enumerate}

For statement (2,a): As noted at the beginning of this subsection, $\Aut(G)$ is an elementary abelian $2$-group, so $G$ is certainly a Miller group. By Proposition \ref{millerGroupProp}(1), $G$ is nilpotent, so we can write $G=\prod_p{G_p}$ where the index $p$ ranges over the primes and $G_p$ denotes the Sylow $p$-subgroup of $G$. By Lemma \ref{simpleLem}(3), $2=\maol(G)=\prod_p{\maol(G_p)}$, and so $\maol(G_p)=2$ for exactly one prime $p$, and $\maol(G_{\ell})=1$ for all primes $\ell\not=p$. We claim that $G$ is a $p$-group. Indeed, otherwise, in view of Proposition \ref{maol1Prop}, $|G|$ has exactly two distinct prime divisors, and, more precisely, $p>2$ and $\pi(G)=\{2,p\}$, with $G_2\cong\IZ/2\IZ$. Since $G$ is nonabelian, it follows that $G_p$ is nonabelian, whence $G_p$ has an (inner) automorphism of order $p$, which implies that $\maol(G)\geqslant\maol(G_p)\geqslant p>2$, a contradiction. So $G$ is indeed a $p$-group for some prime $p$, and again, since $G$ is nonabelian, $2=\maol(G)\geqslant p$, whence $p=2$. This concludes the proof of statement (2,a).

For statement (2,b): Assume that $\zeta G$ is not cyclic, so that we have an embedding $\iota:(\IZ/2\IZ)^2\hookrightarrow\zeta G$. Since $G$ is of nilpotency class $2$ (by Proposition \ref{millerGroupProp}(1)), the central quotient $G/\zeta G$ is an abelian $2$-group, whence we also have a projection $\pi:G/\zeta G \twoheadrightarrow\IZ/2\IZ$. There are four distinct homomorphisms $\varphi:\IZ/2\IZ\rightarrow(\IZ/2\IZ)^2$, and by composition, we get four distinct homomorphisms
\[
f:G\overset{\textrm{can.}}\twoheadrightarrow G/\zeta G\overset{\pi}\twoheadrightarrow\IZ/2\IZ\overset{\varphi}\rightarrow(\IZ/2\IZ)^2\overset{\iota}\hookrightarrow\zeta G.
\]
For each such homomorphism $f:G\rightarrow\zeta G$, we have that $\zeta G\leqslant\ker(f)$, and so the neutral element $1_G$ is the only element of $\zeta G$ inverted by $f$. We may thus consider the associated central automorphism $\alpha_f:G\rightarrow G,g\mapsto gf(g)$. Now fix an element $g\in G$ outside the (index $2$) kernel of the composition
\[
G\overset{\textrm{can.}}\twoheadrightarrow G/\zeta G\overset{\pi}\twoheadrightarrow\IZ/2\IZ.
\]
Then the images of $g$ under the four mentioned central automorphisms $\alpha_f$ are pairwise distinct, which implies that $\maol(G)\geqslant4$, a contradiction. This concludes the proof of statement (2,b).

For statement (2,c): Recall that by Proposition \ref{millerGroupProp}(1), $G$ is nilpotent of class $2$, whence $G'\leqslant\zeta G$, and so $G'$ is cyclic by statement (2,b). Proposition \ref{millerGroupProp}(3) now implies that $|G'|=2$, as required.

For statement (2,d): Assume, aiming for a contradiction, that $|\zeta G|=2$. Then, since $G$ is nilpotent of class $2$ by Proposition \ref{millerGroupProp}(1), we have $G'=\zeta G\cong\IZ/2\IZ$, so that $G$ is an extraspecial $2$-group. By \cite[Theorem 1(c)]{Win72a}, the induced action of $\Aut(G)$ on $G/\zeta G\cong\IF_2^{2n}$ corresponds to the one of an orthogonal group $O^{\epsilon}_{2n}(2)$, for some $\epsilon\in\{+,-\}$ (depending on the isomorphism type of $G$). In any case, this implies that $3\mid|\Aut(G)|$, so that $\Aut(G)$ contains an element of order $3$ by Cauchy's theorem and thus $\maol(G)\geqslant3$, a contradiction. This concludes the proof of statement (2,d).

For statement (2,e): Assume, aiming for a contradiction, that $G/G'$ is \emph{not} an elementary abelian $2$-group. Then we have a projection $\pi:G/G'\twoheadrightarrow\IZ/4\IZ$. There are four endomorphisms $\varphi$ of $\IZ/4\IZ$, and by composition, we obtain four distinct homomorphisms
\[
f:G\overset{\textrm{can.}}\twoheadrightarrow G/G'\overset{\pi}\twoheadrightarrow\IZ/4\IZ\overset{\varphi}\rightarrow\IZ/4\IZ\hookrightarrow\zeta G.
\]
By the three facts that $G'$ is nontrivial, $G'\leqslant\zeta G$ and $\zeta G$ is a cyclic $2$-group, each such homomorphism $f:G\rightarrow\zeta G$ has the property that any nontrivial element of $\zeta G$ is mapped under $f$ to an element of smaller order; in particular, $1_G$ is the only element of $\zeta G$ which is inverted by $f$. It follows that each such homomorphism $f$ induces a central automorphism $\alpha_f:G\rightarrow G,g\mapsto gf(g)$, and any element $g\in G$ which gets mapped under the composition
\[
G\overset{\textrm{can.}}\twoheadrightarrow G/G'\overset{\pi}\twoheadrightarrow\IZ/4\IZ
\]
to a generator of $\IZ/4\IZ$ assumes four distinct images under these central automorphisms $\alpha_f$. It follows that $\maol(G)\geqslant4$, a contradiction, which concludes the proof of statement (2,e).
\end{proof}

We are now ready to classify the finite groups $G$ with $\maol(G)=2$:

\begin{propposition}\label{maol2Prop}
Let $G$ be a finite group. The following are equivalent:
\begin{enumerate}
\item $\maol(G)=2$.
\item $G\cong\IZ/m\IZ$ for some $m\in\{3,4,6\}$.
\end{enumerate}
\end{propposition}

\begin{proof}
The implication \enquote{(2) $\Rightarrow$ (1)} is easy, so we focus on proving \enquote{(1) $\Rightarrow$ (2)}. By Lemma \ref{maol2Lem}(1), it suffices to show that $G$ is abelian. So, working toward a contradiction, let us assume that $G$ is nonabelian. Then we can use all the structural information on $G$ displayed in Lemma \ref{maol2Lem}(2).

Write $G/G'\cong(\IZ/2\IZ)^d$ with $d\in\IN^+$. Note that $d\geqslant3$, as otherwise, $|G|=|G'|\cdot|G/G'|\leqslant 2\cdot 4=8$, which implies that $|\zeta G|=2$, contradicting Lemma \ref{maol2Lem}(2,d). Let us call a $d$-tuple $(g_1,\ldots,g_d)\in G^d$ a \emph{standard tuple in $G$} if and only if it projects to an $\IF_2$-basis of $G/G'$ under the canonical projection $G\twoheadrightarrow G/G'$ (we remark that this notion of a \enquote{standard tuple} will also appear in the next subsection, Subsection \ref{subsec2P3}, and it will be introduced and studied in greater generality in Section \ref{sec4}).

Since $|G'|=2$, the number of standard tuples in $G$ is exactly
\[
2^d\cdot\prod_{i=0}^{d-1}{(2^d-2^i)}.
\]
For each standard tuple $(g_1,\ldots,g_d)$ in $G$, the associated \emph{power-commutator tuple} is defined to be the following $(d+{d \choose 2})$-tuple with entries in $G'$:
\[
(g_1^2,g_2^2,\ldots,g_d^2,[g_1,g_2],[g_1,g_3],\ldots,[g_1,g_d],[g_2,g_3],[g_2,g_4],\ldots,[g_2,g_d],\ldots,[g_{d-1},g_d]).
\]
We say that two standard tuples in $G$ are \emph{equivalent} if and only if they have the same power-commutator tuple. Now, if $c$ denotes the unique nontrivial element of $G'$, then for every standard tuple $(g_1,\ldots,g_d)$ in $G$, the $(d+1)$-tuple $(g_1,\ldots,g_d,c)$ is a polycyclic generating sequence of $G$. Moreover, if $(h_1,\ldots,h_d)$ is a standard tuple in $G$ which is equivalent to $(g_1,\ldots,g_d)$, then the two polycyclic generating sequences $(g_1,\ldots,g_d,c)$ and $(h_1,\ldots,h_d,c)$ of $G$ induce the same power-commutator presentation of $G$ (this is because since $c$ is central in $G$, one has $[g_i,c]=[h_i,c]=1$ for all $i=1,\ldots,d$), so that there exists an automorphism $\alpha$ of $G$ with $g_i^{\alpha}=h_i$ for $i=1,\ldots,d$. This shows that equivalent standard tuples lie in the same orbit of the component-wise action of $\Aut(G)$ on $G^d$.

Note that since $|G'|=2$, the number of distinct power-commutator tuples of standard tuples in $G$, and thus the number of equivalence classes of standard tuples in $G$, is at most $2^{d+{d \choose 2}}$. It follows that there is an equivalence class of standard tuples in $G$ which is of size at least
\begin{align*}
&\frac{2^d\cdot\prod_{i=0}^{d-1}{(2^d-2^i)}}{2^{d+{d \choose 2}}}=\frac{\prod_{i=0}^{d-1}{(2^d-2^i)}}{2^{{d \choose 2}}}=\frac{2^{0+1+2+\cdots+d-1}\cdot\prod_{i=0}^{d-1}{(2^{d-i}-1)}}{2^{{d \choose 2}}}= \\
&\prod_{i=0}^{d-1}{(2^{d-i}-1)}=\prod_{j=1}^d{(2^j-1)}.
\end{align*}
In particular, the component-wise action of $\Aut(G)$ on $G^d$ has an orbit of length at least $\prod_{j=1}^d{(2^j-1)}$. However, since $\maol(G)=2$, no orbit of the action of $\Aut(G)$ on $G^d$ can be of length larger than $2^d$. It follows that
\[
2^d\geqslant\prod_{j=1}^d{(2^j-1)},
\]
which does not hold for any $d\geqslant3$, a contradiction.
\end{proof}

\subsection{Finite groups \texorpdfstring{$G$}{G} with \texorpdfstring{$\maol(G)=3$}{maol(G)=3}}\label{subsec2P3}

We begin by proving some properties of finite groups $G$ with $\maol(G)=3$ which will be crucial for the subsequent discussion:

\begin{lemmma}\label{maol3Lem1}
Let $G$ be a finite group with $\maol(G)=3$. Then the following hold:
\begin{enumerate}
\item If $G$ is abelian, then $G\cong(\IZ/2\IZ)^2$.
\item If $G$ is nonabelian, then the following hold:
\begin{enumerate}
\item The set of element orders of $\Aut(G)$ is contained in $\{1,2,3\}$ (in particular, $\Aut(G)$ is solvable).
\item $G$ is a $\{2,3\}$-group (in particular, $G$ is solvable).
\end{enumerate}
\end{enumerate}
\end{lemmma}

\begin{proof}
For statement (1): Write $G=\prod_p{G_p}$ where $p$ ranges over the primes and $G_p$ denotes the Sylow $p$-subgroup of $G$. If $G_p$ is nontrivial for some prime $p\geqslant5$, then by Lemma \ref{simpleLem}(1,2),
\[
\maol(G)\geqslant\maol(G_p)\geqslant\phi(p)=p-1\geqslant4>3,
\]
a contradiction. Hence $G=G_2\times G_3$, and by Lemma \ref{simpleLem}(3), we have $3=\maol(G)=\maol(G_2)\cdot\maol(G_3)$. We distinguish two cases:
\begin{enumerate}
\item Case: $\maol(G_2)=3$ and $\maol(G_3)=1$. Then by Proposition \ref{maol1Prop}, $G_3$ is trivial, and so $G$ is an abelian $2$-group. By Lemma \ref{simpleLem}(2), we have $3=\maol(G)\geqslant\phi(\Exp(G))$, which implies that $\Exp(G)\in\{2,4\}$. Distinguish two subcases:
\begin{enumerate}
\item Subcase: $\Exp(G)=2$. Then $G\cong(\IZ/2\IZ)^d$ for some positive integer $d$, and we have $3=\maol(G)=2^d-1$. It follows that $d=2$, i.e., $G\cong(\IZ/2\IZ)^2$.
\item Subcase: $\Exp(G)=4$. Then $G\cong(\IZ/2\IZ)^{d_1}\times(\IZ/4\IZ)^{d_2}$ for some $d_1\in\IN$ and some $d_2\in\IN^+$. If $d_2\geqslant2$, then by Lemma \ref{simpleLem}(1), $\maol(G)\geqslant\maol((\IZ/4\IZ)^2)=12>3$, a contradiction. Hence $d_2=1$. If $d_1=0$, then $\maol(G)=\maol(\IZ/4\IZ)=2<3$, a contradiction. Hence $d_1\geqslant1$, which implies by Lemma \ref{simpleLem}(1) that $\maol(G)\geqslant\maol(\IZ/2\IZ\times\IZ/4\IZ)=4>3$, another contradiction.
\end{enumerate}
\item Case: $\maol(G_2)=1$ and $\maol(G_3)=3$. By Lemma \ref{simpleLem}(2), $3=\maol(G_3)\geqslant\phi(\Exp(G_3))$, whence $\Exp(G_3)=3$. It follows that $G_3\cong(\IZ/3\IZ)^d$ for some positive integer $d$, and therefore $\maol(G_3)=3^d-1$, which is never equal to $3$, a contradiction.
\end{enumerate}

For statement (2,a): The \enquote{in particular} follows from Burnside's $p^aq^b$-theorem, so we focus on proving the main assertion. Let $\alpha$ be an automorphism of $G$. For each positive integer $k$, consider the subgroup
\[
\C_G(\alpha^k)=\{g\in G\mid g^{\alpha^k}=g\}\leqslant G.
\]
Since $\maol(G)=3$, all cycles of $\alpha$ on $G$ are of one of the lengths $1$, $2$ or $3$. Equivalently, $G=\C_G(\alpha^2)\cup \C_G(\alpha^3)$. But no finite group is a union of two proper subgroups (see e.g.~\cite[Exercise 1.3.9, p.~17]{Rob96a}). It follows that either $\C_G(\alpha^2)=G$ or $\C_G(\alpha^3)=G$, and accordingly, that the order of $\alpha$ divides $2$ or $3$, which concludes the proof of statement (2,a).

For statement (2,b): By statement (2,a), $\Aut(G)$ is solvable. It follows that $G$, being an extension of the solvable group $\Inn(G)$ by the abelian group $\zeta G$, is also solvable. By assumption, all conjugacy classes in $G$ are of one of the lengths $1$, $2$ or $3$, and thus all element centralizers in $G$ are of one of the indices $1$, $2$ or $3$. It follows that the central quotient $G/\zeta G$ is a $\{2,3\}$-group. Since $G$ is solvable, $G$ has a Hall $\{2,3\}'$-subgroup $G_{\{2,3\}'}$, which must be central and thus normal (or, equivalently, unique). Moreover, $G$ has a Hall $\{2,3\}$-subgroup $G_{\{2,3\}}$, which, being centralized by $G_{\{2,3\}'}$, is also normal, and so we have $G=G_{\{2,3\}}\times G_{\{2,3\}'}$. If $G_{\{2,3\}'}$ was nontrivial, it would follow by Lemma \ref{simpleLem}(2) that
\[
3=\maol(G)\geqslant\maol(G_{\{2,3\}'})\geqslant\phi(\Exp(G_{\{2,3\}'}))\geqslant 4,
\]
a contradiction. Hence $G=G_{\{2,3\}}$, which concludes our proof of statement (2,b).
\end{proof}

Note that according to Theorem \ref{mainTheo}(1), the only nonabelian finite group $G$ with $\maol(G)=3$ is $G=\Sym(3)$, for which the set of orders of inner automorphisms is $\{1,2,3\}$. It is precisely this property which we will show next for all nonabelian finite groups $G$ with $\maol(G)=3$:

\begin{lemmma}\label{maol3Lem2}
Let $G$ be a nonabelian finite group with $\maol(G)=3$. Then the set of orders of inner automorphisms of $G$ is $\{1,2,3\}$ (and hence the set of orders of all automorphisms of $G$ is also $\{1,2,3\}$).
\end{lemmma}

\begin{proof}
The \enquote{and hence} follows from Lemma \ref{maol3Lem1}(2,a), so we focus on proving the main assertion. Note that by Lemma \ref{maol3Lem1}(2,a), the set of orders of inner automorphisms of $G$ is contained in $\{1,2,3\}$, whence it suffices to show that $\Exp(\Inn(G))$ can neither be $2$ nor $3$. Assume otherwise. Then $\Inn(G)=G/\zeta G$ is of prime-power order, and thus $G$ is nilpotent. Hence, in view of Lemma \ref{maol3Lem1}(2,b), we have $G=G_2\times G_3$, where $G_p$ denotes the Sylow $p$-subgroup of $G$ for $p\in\{2,3\}$. By Lemma \ref{simpleLem}(3), it follows that $3=\maol(G)=\maol(G_2)\cdot\maol(G_3)$. Distinguish two cases:
\begin{enumerate}
\item Case: $\maol(G_2)=3$ and $\maol(G_3)=1$. Then by Proposition \ref{maol1Prop}, $G_3$ is trivial, and so $G$ is a nonabelian $2$-group with $\maol(G)=3$. All non-central conjugacy classes in $G$ are of length $2$ and are $\Aut(G)$-orbits (otherwise, there would be an $\Aut(G)$-orbit of length at least $2\cdot 2=4>3$). It follows that for any $\alpha\in\Aut(G)$, the subgroup
\[
\C_G(\alpha^2)=\{g\in G\mid g^{\alpha^2}=g\}
\]
contains the generating set $G\setminus\zeta G$, and thus $\C_G(\alpha^2)=G$. Therefore, $\alpha^2=\id_G$, whence $\Exp(\Aut(G))=2$. However, we are assuming that $\maol(G)=3$, and so by the orbit-stabilizer theorem, $3\mid|\Aut(G)|$, so that $\Aut(G)$ contains an order $3$ element by Cauchy's theorem, a contradiction.
\item Case: $\maol(G_2)=1$ and $\maol(G_3)=3$. Then by Proposition \ref{maol1Prop}, $G_2$ is abelian, whence $G_3$ is a nonabelian $3$-group with $\maol(G_3)=3$. Observe that all non-central conjugacy classes in $G_3$ are of length $3$ and are $\Aut(G_3)$-orbits. Since the Sylow $3$-subgroup of $\Sym(3)$ is abelian, it follows that any two inner automorphisms of $G_3$ commute on $G_3\setminus\zeta G_3$, which is a generating subset for $G_3$, so that $\Inn(G_3)$ is abelian, i.e., $G_3$ is nilpotent of class $2$. We now list some more structural properties of $G_3$, in the spirit of Lemma \ref{maol2Lem}(2):
\begin{itemize}
\item $\zeta G_3$ is cyclic. Indeed, otherwise, a suitable non-central element of $G_3$ would have at least $|\Hom(\IZ/3\IZ,(\IZ/3\IZ)^2)|=9$ distinct images under central automorphisms of $G_3$, a contradiction.
\item $|\zeta G_3|>3$. Indeed, otherwise, $G_3$ would be an extraspecial $3$-group. If $|G_3|=3^{1+2}$, then one can check with GAP \cite{GAP4} that $\maol(G)\in\{18,24\}$, a contradiction. And if $|G_3|=3^{1+2n}$ with $n\geqslant2$, then by \cite[Theorem 1]{Win72a} (and the fact that the symplectic group $\Sp_{2n}(q)\leqslant\GL_{2n}(q)$ acts transitively on $\IF_q^{2n}\setminus\{0\}$, which can be derived from Witt's theorem), $\Aut(G)$ has an orbit of length at least $3^{2n-2}-1\geqslant 8>3$ on $G$, a contradiction.
\item $G_3/\zeta G_3$ is an elementary abelian $3$-group. Indeed, otherwise, a suitable non-central element of $G_3$ would have at least $|\End(\IZ/9\IZ)|=9$ distinct images under central automorphisms of $G_3$, a contradiction.
\item $G_3'\cong\IZ/3\IZ$. Since $G_3$ is nilpotent of class $2$, $G_3'\leqslant\zeta G_3$, whence $G_3'$ is cylic. Moreover, since $G_3$ is nilpotent of class $2$ (which implies that $[x^e,y^f]=[x,y]^{ef}$ for all $x,y\in G_3$ and all $e,f\in\IZ$) and $\Exp(G_3/\zeta G_3)=3$, the exponent of $G_3'$ must be $3$ as well.
\item $G_3/G_3'$ is an elementary abelian $3$-group. Indeed, otherwise, a suitable element of $G_3$ outside $G_3'$ would have at least $|\End(\IZ/9\IZ)|=9$ distinct images under central automorphisms of $G_3$ (note that any homomorphism $f:G_3\rightarrow\zeta G_3$ has the property that $1_{G_3}$ is the only element of $\zeta G_3$ mapped to its inverse by $f$, since $\ker(f)$ contains $G_3'$, which is a nontrivial subgroup of the cyclic $3$-group $\zeta G_3$), a contradiction.
\end{itemize}
We can now repeat the \enquote{standard tuples} argument from the proof of Proposition \ref{maol2Prop} almost verbatim (only needing to replace the prime $2$ by $3$) and find that with $G_3/G_3'\cong(\IZ/3\IZ)^d$, we necessarily have
\[
3^d\geqslant\prod_{j=1}^d{(3^j-1)},
\]
which implies that $d=1$ and thus $|G_3|=3^2$, contradicting that $G_3$ is nonabelian.
\end{enumerate}
\end{proof}

We are now ready to prove the main result of this subsection:

\begin{propposition}\label{maol3Prop}
Let $G$ be a finite group. The following are equivalent:
\begin{enumerate}
\item $\maol(G)=3$.
\item $G$ is isomorphic to $(\IZ/2\IZ)^2$ or to $\Sym(3)$.
\end{enumerate}
\end{propposition}

\begin{proof}
The implication \enquote{(2) $\Rightarrow$ (1)} is easy, so we focus on proving \enquote{(1) $\Rightarrow$ (2)}. So, assume that $G$ is a finite group with $\maol(G)=3$. If $G$ is abelian, then $G\cong(\IZ/2\IZ)^2$ by Lemma \ref{maol3Lem1}(1). We may thus assume that $G$ is nonabelian and need to show that $G\cong\Sym(3)$.

By Lemma \ref{maol3Lem2}, the set of element orders of $\Inn(G)$ is $\{1,2,3\}$, and hence by \cite[Theorem]{BS91a}, $\Inn(G)\cong G/\zeta G$ is one of the following:
\begin{itemize}
\item the Frobenius group $(\IZ/3\IZ)^d\rtimes\IZ/2\IZ=\Dih((\IZ/3\IZ)^d)$ for some $d\in\IN^+$, or
\item the Frobenius group $(\IZ/2\IZ)^{2d}\rtimes\IZ/3\IZ$ for some $d\in\IN^+$.
\end{itemize}
Note that the maximum conjugacy class length in $\Inn(G)$ cannot exceed the maximum conjugacy class length in $G$, which is at most $3$. But in the Frobenius group $(\IZ/2\IZ)^{2d}\rtimes\IZ/3\IZ$, the length of the conjugacy class of any generator of the Frobenius complement $\IZ/3\IZ$ is $2^{2d}\geqslant 4>3$. Hence $\Inn(G)\cong\Dih((\IZ/3\IZ)^d)$, which has a conjugacy class of length $3^d$, so that $d=1$ and thus
\[
G/\zeta G\cong \Inn(G)\cong\Dih(\IZ/3\IZ)\cong\Sym(3).
\]
Next, we claim that $\zeta G$ is a $3$-group. Indeed, otherwise, in view of Lemma \ref{maol3Lem1}(2,b), there is an embedding $\IZ/2\IZ\overset{\iota}\hookrightarrow\zeta G$. Moreover, we have a finite sequence of group homomorphisms
\[
G\overset{\textrm{can.}}\twoheadrightarrow G/\zeta G\overset{\sim}\rightarrow\Sym(3)\overset{\pi}\twoheadrightarrow\IZ/2\IZ,
\]
and through composition, we obtain a nontrivial homomorphism $G\overset{f}\rightarrow\zeta G$ with nontrivial associated central automorphism $\alpha_f$. Now, let
\[
g \in G \setminus \ker\left(G\overset{\textrm{can.}}\twoheadrightarrow G/\zeta G\overset{\sim}\rightarrow\Sym(3)\overset{\pi}\twoheadrightarrow\IZ/2\IZ\right).
\]
Then the conjugacy class length of the image of $g$ in $G/\zeta G\cong\Sym(3)$ is $3$, whence $g^G$ meets three distinct cosets of $\zeta G$ in $G$. Moreover, $g^{\alpha_f}$ is an element in the same central coset as $g$, but distinct from $g$ itself. It follows that $|g^G|\geqslant 2\cdot 3=6>3$, a contradiction. This concludes our argument that $\zeta G$ is a $3$-group.

It now follows that $G$ has a normal, abelian Sylow $3$-subgroup $G_3$, and $G=G_3\rtimes\IZ/2\IZ$, where the generator $h$ of $\IZ/2\IZ$ centralizes the index $3$ subgroup $\zeta G$ of $G_3$. Let $g\in G_3\setminus\zeta G$. Then, writing $G_3$ additively, we have $g^h=-g+z$ for some $z\in\zeta G$. It follows that
\[
3g=(3g)^h=3g^h=3(-g+z)=-3g+3z,
\]
and thus $3(2g-z)=0$. Through replacing $g$ by $2g-z$, we may assume w.l.o.g.~that $\ord(g)=3$. Recall that $g^h=-g+z$ for some $z\in\zeta G$, and note that $\ord(z)\mid 3$ (otherwise, $g^h$ would have order larger than $3=\ord(g)$, a contradiction). Set $g':=g+z$. Then
\[
(g')^h=g^h+z^h=-g+z+z=-g-z=-g'.
\]
Hence, through replacing $g$ by $g'$, we may assume w.l.o.g.~that $g^h=g^{-1}$. This entails that
\[
G=G_3\rtimes\IZ/2\IZ=(\zeta G\times\langle g\rangle)\rtimes\IZ/2\IZ=\zeta G\times(\langle g\rangle\rtimes\IZ/2\IZ) \cong \zeta G\times\Sym(3).
\]
Therefore, if $\zeta G$ is nontrivial, then by Lemma \ref{simpleLem}(2),
\[
\maol(\zeta G)\geqslant\phi(\Exp(\zeta G))\geqslant\phi(3)=2,
\]
and thus, by Lemma \ref{simpleLem}(1),
\[
3=\maol(G)\geqslant\maol(\zeta G)\cdot\maol(\Sym(3))\geqslant 2\cdot 3=6,
\]
a contradiction. Hence $\zeta G$ is trivial, and so $G\cong\Sym(3)$, as we needed to show.
\end{proof}

\subsection{Proof of Theorem \ref{mainTheo}(1)}\label{subsec2P4}

This is immediate from Propositions \ref{maol1Prop}, \ref{maol2Prop} and \ref{maol3Prop}.

\section{Finite groups \texorpdfstring{$G$}{G} with \texorpdfstring{$\maol(G)=8$}{maol(G)=8}}\label{sec3}

This section is concerned with the proof of Theorem \ref{mainTheo}(2). We begin by introducing a certain infinite sequence of finite $2$-groups:

\begin{definition}\label{2GroupDef}
For $n\in\IN^+$, let $G_n$ be the finite $2$-group given by the power-commutator presentation
\begin{align*}
\langle x_1,\ldots,x_{2^n+1},a,b \mid &[a,b]=[x_i,a]=[x_i,b]=1,[x_{2i-1},x_{2i}]=a,[x_{2i},x_{2i+1}]=b, \\
&[x_i,x_j]=1\text{ if }|i-j|>1,x_1^2=x_{2^n+1}^2=b, \\
&a^2=b^2=x_i^2=1\text{ if }1<i<2^n+1\rangle.
\end{align*}
\end{definition}

\begin{remark}\label{2GroupRem}
We note the following concerning Definition \ref{2GroupDef}:
\begin{enumerate}
\item As is easy to check, $G_n$ is a finite $2$-group of order $2^{2^n+3}$, of nilpotency class $2$ and of exponent $4$. Moreover, $C_n:=\zeta G_n=\langle a,b\rangle\cong(\IZ/2\IZ)^2$ and $Q_n:=G_n/C_n\cong(\IZ/2\IZ)^{2^n+1}$.
\item The specified power-commutator presentation of $G_n$ is inspired by the presentation of the infinite $2$-group
\begin{align*}
\langle a,b,x_1,x_2,\ldots \mid &[a,b]=[x_i,a]=[x_i,b]=1,[x_{2i-1},x_{2i}]=a,[x_{2i},x_{2i+1}]=b, \\
&[x_i,x_j]=1\text{ if }|i-j|>1,x_1^2=b,a^2=b^2=x_i^2=1\text{ if }i>1\rangle,
\end{align*}
which was given by Robinson and Wiegold as an example of a group with infinite automorphism group but largest automorphism orbit length $4$, see \cite[Section 4, Construction before Proposition 3, Proposition 3 itself and Remark after the proof of Proposition 3]{RW84a}.
\item The assignment $a\mapsto a$, $b\mapsto b$, $x_i\mapsto x_i$ for $i\not=2^n$ and $x_{2^n}\mapsto x_{2^n}x_{2^n+1}$ extends to a non-central automorphism $\alpha_n$ of $G_n$.
\end{enumerate}
\end{remark}

We prove Theorem \ref{mainTheo}(2) by proving the following proposition:

\begin{proposition}\label{2GroupProp}
For all $n\in\IN^+$, we have $|\Aut(G_n):\Aut_{\cent}(G_n)|=2$. In particular, $\maol(G_n)=8$.
\end{proposition}

\begin{proof}
The \enquote{In particular} follows from the observations that $\Aut_{\cent}(G_n)$ acts transitively on each nontrivial coset of $C_n$ and that $|C_n|=4$ (see Remark \ref{2GroupRem}(1)). As for the main assertion, we will show that every automorphism of $G_n$ lies in the coset union $\Aut_{\cent}(G_n) \cup \Aut_{\cent}(G_n)\alpha_n$, with $\alpha_n$ as in Remark \ref{2GroupRem}(3). We do so in several steps, in each of which we first make a claim (surrounded by a frame), which is subsequently proved.

\begin{mdframed}
Claim 1: Let $g\in G_n$, and assume that $\ord(g)=4$ and $|G_n:\C_{G_n}(g)|=2$. Then:
\begin{enumerate}
\item If there is no $h\in G_n$ with $g^2=[g,h]$, then $g\in x_1C_n$.
\item If there is an $h\in G_n$ with $g^2=[g,h]$, then $g\in x_{2^n+1}C_n$.
\end{enumerate}
In particular, the two subgroups $\langle x_1,C_n\rangle$ and $\langle x_{2^n+1},C_n\rangle$ are characteristic in $G_n$.
\end{mdframed}

Write $g\equiv x_1^{u_1}\cdots x_k^{u_k}\Mod{C_n}$ where $u_i\in\IZ/2\IZ$ and $u_k\not=0$. If $k=1$, then both asserted implications are true; the former because its necessary condition is true, and the latter because its sufficient condition is false ($g^2=x_1^2=b$, but $[g,G_n]=[x_1,G_n]=\{1,a\}$). So assume henceforth that $k>1$. We make a case distinction.
\begin{enumerate}
\item Case: $k<2^n+1$. This case plays out analagously to the proof of statement (i) in \cite[proof of Proposition 3]{RW84a}, but we will give the argument here for completeness and the reader's convenience. If $h\in\C_{G_n}(g)$, then modulo elements known to commute with $g$, we can write $h=x_1^{v_1}\cdots x_{k+1}^{v_{k+1}}$ with $v_i\in\IZ/2\IZ$. Then the assumption that $1=[g,h]$ yields
\begin{equation}\label{eq1}
1=[x_1,x_2]^{u_1v_2+u_2v_1}[x_2,x_3]^{u_2v_3+u_3v_2}\cdots[x_{k-1},x_k]^{u_{k-1}v_k+u_kv_{k-1}}[x_k,x_{k+1}]^{u_kv_{k+1}}.
\end{equation}
Now use the commutator relations to equivalently rewrite Formula (\ref{eq1}) into a pair of linear equations over $\IF_2$ in the variables $v_1,\ldots,v_{k+1}$. The final terms of these equations look as follows:
\begin{align*}
\ldots+u_kv_{k-1}+u_{k-1}v_k\hspace{1.6cm}=&0, \\
\ldots\ldots\ldots\ldots\ldots\ldots\ldots+u_kv_{k+1}=&0.
\end{align*}
Since $u_k\not=0$ by assumption and $v_{k+1}$ does not occur in the first equation, the two equations are $\IF_2$-linearly independent, which implies that $|G_n:\C_{G_n}(g)|=2^2=4>2$, a contradiction.
\item Case: $k=2^n+1$. If $h\in\C_{G_n}(g)$, then modulo elements known to commute with $g$ (namely $a$ and $b$), we can write $h=x_1^{v_1}\cdots x_{2^n+1}^{v_{2^n+1}}$ with $v_i\in\IZ/2\IZ$. The assumption that $1=[g,h]$ yields
\begin{equation}\label{eq2}
1=[x_1,x_2]^{u_1v_2+u_2v_1}[x_2,x_3]^{u_2v_3+u_3v_2}\cdots[x_{2^n},x_{2^n+1}]^{u_{2^n}v_{2^n+1}+u_{2^n+1}v_{2^n}}.
\end{equation}
Using the commutator relations, we can equivalently rewrite Formula (\ref{eq2}) into the following pair of linear equations over $\IF_2$ in the variables $v_1,\ldots,v_{2^n+1}$:
\begin{align*}
u_2v_1+u_1v_2+u_4v_3+u_3v_4+\ldots\ldots\ldots\ldots+u_{2^n}v_{2^n-1}+u_{2^n-1}v_{2^n}\hspace{2.04cm}&=0, \\
u_3v_2+u_2v_3+u_5v_4+u_4v_5+\ldots+u_{2^n-2}v_{2^n-1}+u_{2^n+1}v_{2^n}+u_{2^n}v_{2^n+1}&=0.
\end{align*}
We make a subcase distinction:
\begin{enumerate}
\item Subcase: At least one of $u_1,u_2,\ldots,u_{2^n}$ is nonzero. Then since $u_{2^n+1}\not=0$ by assumption, both equations are nonzero, and since $|G_n:\C_{G_n}(g)|=2$, they must be $\IF_2$-linearly dependent, which implies that $u_2=u_4=\cdots=u_{2^n}=0$ and $u_1=u_3=\cdots=u_{2^n+1}=1$. We conclude that $g\equiv x_1x_3\cdots x_{2^n+1}\Mod{C_n}$, whence $g^2=x_1^2x_3^2\cdots x_{2^n+1}^2=b\cdot 1 \cdots 1\cdot b=b^2=1$, contradicting our assumption that $\ord(g)=4$.
\item Subcase: $u_1=u_2=\cdots=u_{2^n}=0$. Then $g\equiv x_{2^n+1}\Mod{C_n}$. Similarly to the argument for $k=1$ above, we find that both asserted implications are true; the former because its sufficient condition is false ($g^2=x_{2^n+1}^2=b=[x_{2^n},x_{2^n+1}]=[x_{2^n},g]$), and the latter because its necessary condition is true.
\end{enumerate}
\end{enumerate}

\begin{mdframed}
Claim 2: Each of the three central order $2$ subgroups $\langle a\rangle$, $\langle b\rangle$ and $\langle ab\rangle$ is characteristic in $G_n$.
\end{mdframed}

Indeed, we have that
\begin{itemize}
\item $\langle a\rangle=[\langle x_1,C_n\rangle,G_n]$,
\item $\langle b\rangle=\mho^1(\langle x_1,C_n\rangle)=\langle\{g^2\mid g\in\langle x_1,C_n\rangle\}\rangle$ and
\item $\langle ab\rangle=\langle C_n\setminus(\langle a\rangle \cup \langle b\rangle)\rangle$.
\end{itemize}

\begin{mdframed}
Claim 3: For each $m\in\{0,\ldots,2^{n-1}\}$, the subgroup $\langle C_n,x_1,x_3,\ldots,x_{2m+1}\rangle$ is characteristic in $G_n$.
\end{mdframed}

We proceed by induction on $m$. The induction base, $m=0$, is clear by Claim 1. So assume that $m\geqslant1$, and that $\langle C_n,x_1,x_3,\ldots,x_{2m-1}\rangle$ is characteristic in $G_n$. Note: If $m=2^{n-1}$, then we are done by Claim 1, as
\begin{align*}
\langle C_n,x_1,x_3,\ldots,x_{2m+1}\rangle &= \langle\langle C_n,x_1,x_3,\ldots,x_{2m-1}\rangle, \langle C_n,x_{2m+1}\rangle\rangle \\
&= \langle\langle C_n,x_1,x_3,\ldots,x_{2m-1}\rangle, \langle C_n,x_{2^n+1}\rangle\rangle.
\end{align*}
We may thus also assume that $m<2^{n-1}$. Set
\[
H_m:=\C_{G_n}(\langle C_n,x_1,x_3,\ldots,x_{2m-1}\rangle)=\langle C_n,x_1,x_3,\ldots,x_{2m-1},x_{2m+1},x_{2m+2},\ldots,x_{2^n+1}\rangle.
\]
Note that by the induction hypothesis, $H_m$ is characteristic in $G_n$ and
\[
\zeta H_m=\langle C_n,x_1,x_3,\ldots,x_{2m-1}\rangle.
\]
We will show the following claim: \enquote{If $g\in H_m$, $[g,H_m]=\langle a\rangle$ and $|H_m:\C_{H_m}(g)|=2$, then $g\in x_{2m+1}\zeta H_m$.}

Note that since
\[
\langle C_n,x_1,x_3,\ldots,x_{2m+1}\rangle = \langle \zeta H_m, x_{2m+1}\zeta H_m\rangle,
\]
once this claim is proved, our inductive proof of Claim 3 is complete.

The proof of the claim is similar to the argument for Claim 1: Write
\[
g\equiv x_{2m+1}^{u_{2m+1}}\cdots x_k^{u_k}\Mod{\zeta H_m}
\]
with $u_i\in\IZ/2\IZ$ and $u_k\not=0$. Note: If $k=2m+1$, then the asserted implication is true because its necessary condition is true. So we may assume that $k>2m+1$. We make a case distinction.
\begin{enumerate}
\item Case: $k<2^n+1$. If $h\in\C_{H_m}(g)$, then modulo elements known to commute with $g$, we can write $h=x_{2m+1}^{v_{2m+1}}\cdots x_{k+1}^{v_{k+1}}$ with $v_i\in\IZ/2\IZ$. Our assumption that $1=[g,h]$ yields
\begin{equation}\label{eq3}
1=[x_{2m+1},x_{2m+2}]^{u_{2m+1}v_{2m+2}+u_{2m+2}v_{2m+1}}\cdots[x_{k-1},x_k]^{u_{k-1}v_k+u_kv_{k-1}}[x_k,x_{k+1}]^{u_kv_{k+1}}.
\end{equation}
Using the commutator relations, we can equivalently rewrite Formula (\ref{eq3}) into a pair of linear equations over $\IF_2$, which look like this:
\begin{align*}
\ldots+u_kv_{k-1}+u_{k-1}v_k\hspace{1.6cm}=&0, \\
\ldots\ldots\ldots\ldots\ldots\ldots\ldots+u_kv_{k+1}=&0.
\end{align*}
Since $u_k\not=0$, these two equations are $\IF_2$-linearly independent, which implies that $|H_m:\C_{H_m}(g)|=2^2=4>2$, a contradiction.
\item Case: $k=2^n+1$. If $h\in\C_{H_m}(g)$, then modulo elements known to commute with $g$, we can write $h=x_{2m+1}^{v_{2m+1}}\cdots x_{2^n+1}^{v_{2^n+1}}$ with $v_i\in\IZ/2\IZ$. Our assumption that $1=[g,h]$ yields
\begin{equation}\label{eq4}
1=[x_{2m+1},x_{2m+2}]^{u_{2m+1}v_{2m+2}+u_{2m+2}v_{2m+1}}\cdots[x_{2^n},x_{2^n+1}]^{u_{2^n}v_{2^n+1}+u_{2^n+1}v_{2^n}}.
\end{equation}
Using the commutator relations, we can equivalently rewrite Formula (\ref{eq4}) into the following pair of linear equations over $\IF_2$:
\begin{align*}
u_{2m+2}v_{2m+1}+u_{2m+1}v_{2m+2}+u_{2m+4}v_{2m+3}+.\hspace{0.015cm}\ldots\hspace{0.015cm}.+u_{2^n}v_{2^n-1}+u_{2^n-1}v_{2^n}\hspace{2cm}&=0, \\
u_{2m+3}v_{2m+2}+u_{2m+2}v_{2m+3}+\ldots+u_{2^n-2}v_{2^n-1}+u_{2^n+1}v_{2^n}+u_{2^n}v_{2^n+1}&=0.
\end{align*}
We make a subcase distinction:
\begin{enumerate}
\item Subcase: At least one of $u_{2m+1},\ldots,u_{2^n}$ is nonzero. Then since $u_{2^n+1}\not=0$ by assumption, both equations are nonzero, and since $|H_m:\C_{H_m}(g)|=2$, the equations must be $\IF_2$-linearly dependent. It follows that $u_{2m+2}=u_{2m+4}=\cdots=u_{2^n}=0$ and $u_{2m+1}=u_{2m+3}=\cdots=u_{2^n+1}=1$. Hence
\[
g\equiv x_{2m+1}x_{2m+3}\cdots x_{2^n+1}\Mod{\zeta H_m},
\]
and therefore
\[
[g,x_{2m+2}]=[x_{2m+1},x_{2m+2}]\cdot [x_{2m+2},x_{2m+3}]=a\cdot b\not=a,
\]
contradicting our assumption that $[g,H_m]=\langle a\rangle$.
\item Subcase: $u_{2m+1}=u_{2m+2}=\cdots=u_{2^n}=0$. Then $g\equiv x_{2^n+1}\Mod{\zeta H_m}$, and thus
\[
[g,H_m]=\langle[x_{2^n},x_{2^n+1}]\rangle=\langle b\rangle,
\]
contradicting our assumption that $[g,H_m]=\langle a\rangle$.
\end{enumerate}
\end{enumerate}

In what follows, $\alpha$ is an arbitrary automorphism of $G_n$. For $i\in\{1,2,\ldots,2^n+1\}$, we can write
\[
x_i^{\alpha}=x_1^{\alpha_{i,1}}x_2^{\alpha_{i,2}}\cdots x_{2^n+1}^{\alpha_{i,2^n+1}}\Mod{C_n}
\]
with $\alpha_{i,j}\in\IZ/2\IZ$.

\begin{mdframed}
Claim 4: The following hold:
\begin{enumerate}
\item $\alpha_{1,1}=1$, and $\alpha_{1,j}=0$ for $j>1$.
\item $\alpha_{2^n+1,2^n+1}=1$, and $\alpha_{2^n+1,j}=0$ for $j<2^n+1$.
\item For $i\in\{1,\ldots,2^{n-1}-1\}$ and $j\in\{1,\ldots,2^{n-1}\}$:
\begin{enumerate}
\item $\alpha_{2i+1,2j}=0$.
\item If $j>i$, then $\alpha_{2i+1,2j+1}=0$.
\item $\alpha_{2i+1,1}=0$, and $\alpha_{2i+1,2i+1}=1$.
\end{enumerate}
\end{enumerate}
\end{mdframed}

Indeed, statements (1) and (2) are clear by Claim 1. Moreover, statements (3,a) and (3,b) are clear by Claim 3. As for statement (3,c), note that if $\alpha_{2i+1,1}=1$, then $\ord(x_{2i+1}^{\alpha})=4$, a contradiction. Finally, $\alpha_{2i+1,2i+1}=1$ since otherwise, by Claim 3,
\[
x_{2i+1}^{\alpha}\in\langle C_n,x_1,x_3,\ldots,x_{2i-1}\rangle,
\]
which contradicts the characteristicity in $G_n$ of $\langle C_n,x_1,x_3,\ldots,x_{2i-1}\rangle$.

\begin{mdframed}
Claim 5: For each $i\in\{1,\ldots,2^{n-1}+1\}$, we have the following:
\begin{enumerate}
\item The subgroup $\langle x_{2i-1},C_n\rangle$ is characteristic in $G_n$.
\item For $j=1,2,\ldots,i-1$, we have
\begin{enumerate}
\item $\alpha_{2j,2j}=1$.
\item $x_{2j}^{\alpha}x_{2j}^{-1}\equiv x_1^{\alpha_{2j,1}}x_3^{\alpha_{2j,3}}\cdots x_{2i-1}^{\alpha_{2j,2i-1}}x_{2i}^{\alpha_{2j,2i}}\cdots x_{2^n+1}^{\alpha_{2j,2^n+1}}\Mod{C_n}$.
\end{enumerate}
\end{enumerate}
\end{mdframed}

This is analogous to the proof of statement (iv) in \cite[proof of Proposition 3]{RW84a}, but we will give the argument in detail here, for the reader's convenience and to make sure it is not a problem that (in contrast to the situation in \cite[proof of Proposition 3]{RW84a}) we do not know at this point whether $x_{2j,1}=0$.

We proceed by induction on $i$. The case \enquote{$i=1$} is clear by Claim 1, and the case \enquote{$i=2$} is clear by Claim 4(3) and the observation that if $\alpha_{2,2}=0$, then $x_2^{\alpha}\in\C_{G_n}(\langle x_1,C_n\rangle)$, which contradicts the characteristicity of $\C_{G_n}(\langle x_1,C_n\rangle)$ in $G_n$.

We may thus assume that $i\geqslant2$, and that the assertion has been proved for $i$. Let $j\in\{1,2,\ldots,i-1\}$. Then by the induction hypothesis,
\[
x_{2j}^{\alpha}\equiv x_{2j}\cdot x_1^{\alpha_{2j,1}}x_3^{\alpha_{2j,3}}\cdots x_{2i-1}^{\alpha_{2j,2i-1}}x_{2i}^{\alpha_{2j,2i}}\cdots x_{2^n+1}^{\alpha_{2j,2^n+1}}\Mod{C_n}
\]
and
\[
x_{2i-1}^{\alpha}\equiv x_{2i-1}\Mod{C_n}.
\]
Hence if $j<i-1$, it follows from $1=[x_{2j},x_{2i-1}]$ that
\[
1=[x_{2i-1},x_{2i}]^{\alpha_{2j,2i}}=a^{\alpha_{2j,2i}},
\]
and thus $\alpha_{2j,2i}=0$. And if $j=i-1$, it follows from
\[
b=[x_{2i-2},x_{2i-1}]=[x_{2j},x_{2i-1}]
\]
that
\[
b=[x_{2i-2},x_{2i-1}]\cdot[x_{2i-1},x_{2i}]^{\alpha_{2j,2i}}=b\cdot a^{\alpha_{2j,2i}},
\]
whence, again, $\alpha_{2j,2i}=0$. We just showed that
\begin{equation}\label{eq5}
\alpha_{2j,2i}=0\text{ for }j=1,2,\ldots,i-1.
\end{equation}
Now, by the induction hypothesis and Formula (\ref{eq5}), we have
\[
x_{2j}^{\alpha}\equiv x_{2j}\cdot x_1^{\alpha_{2j,1}}x_3^{\alpha_{2j,3}}\cdots x_{2i+1}^{\alpha_{2j,2i+1}}x_{2i+2}^{\alpha_{2j,2i+2}}\cdots x_{2^n+1}^{\alpha_{2j,2^n+1}}\Mod{C_n}
\]
and
\[
x_{2i+1}^{\alpha}\equiv x_3^{\alpha_{2i+1,3}}x_5^{\alpha_{2i+1,5}}\cdots x_{2i+1}^{\alpha_{2i+1,2i+1}}\Mod{C_n}.
\]
It follows from $1=[x_{2j},x_{2i+1}]$ that
\begin{align*}
1 &=[x_{2j-1},x_{2j}]^{\alpha_{2i+1,2j-1}}[x_{2j},x_{2j+1}]^{\alpha_{2i+1,2j+1}}[x_{2i+1},x_{2i+2}]^{\alpha_{2j,2i+2}\alpha_{2i+1,2i+1}} \\
&=a^{\alpha_{2i+1,2j-1}+\alpha_{2j,2i+2}\alpha_{2i+1,2i+1}}b^{\alpha_{2i+1,2j+1}},
\end{align*}
which implies that $\alpha_{2i+1,2j+1}=0$. We just showed that
\begin{equation}\label{eq6}
\alpha_{2i+1,2j+1}=0\text{ for }j=1,2,\ldots,i-1.
\end{equation}
Together with Claim 4(3), Formula (\ref{eq6}) implies that $\langle x_{2i+1},C_n\rangle$ is characteristic in $G_n$. Finally, by definition,
\[
x_{2i}^{\alpha}\equiv x_1^{\alpha_{2i,1}}x_2^{\alpha_{2i,2}}\cdots x_{2^n+1}^{\alpha_{2i,2^n+1}}\Mod{C_n},
\]
and by the induction hypothesis,
\[
x_{2j+1}^{\alpha}\equiv x_{2j+1}\Mod{C_n}.
\]
If $j<i-1$, then it follows from $1=[x_{2i},x_{2j+1}]$ that
\[
1=[x_{2j},x_{2j+1}]^{\alpha_{2i,2j}}[x_{2j+1},x_{2j+2}]^{\alpha_{2i,2j+2}}=b^{\alpha_{2i,2j}}a^{\alpha_{2i,2j+2}},
\]
whence $\alpha_{2i,2j}=\alpha_{2i,2j+2}=0$. This shows that
\begin{equation}\label{eq7}
\alpha_{2i,2j}=0\text{ for }j=1,\ldots,i-1.
\end{equation}
To complete the inductive proof of Claim 5, it remains to show that $\alpha_{2i,2i}=1$. Assume otherwise. Then by Formula (\ref{eq7}), $x_{2i}^{\alpha}\in\C_{G_n}(\langle C_n,x_1,x_3,\ldots,x_{2i-1}\rangle)$, which contradicts the characteristicity in $G_n$ of $\C_{G_n}(\langle C_n,x_1,x_3,\ldots,x_{2i-1}\rangle)$.

\begin{mdframed}
Claim 6: For all $i,j\in\{1,2,\ldots,2^{n-1}\}$, we have
\[
\alpha_{2i,2j}=\begin{cases}1, & \text{if }i=j, \\ 0, & \text{otherwise.}\end{cases}
\]
In other words,
\[
x_{2i}^{\alpha}\equiv x_{2i}\cdot x_1^{\alpha_{2i,1}}x_3^{\alpha_{2i,3}}\cdots x_{2^n+1}^{\alpha_{2i,2^n+1}}\Mod{C_n}.
\]
\end{mdframed}

This is immediate from Claim 5 with $i=2^{n-1}+1$.

Note that by Claims 5 and 6, we now know that modulo $C_n$, $\alpha$ fixes each of $x_1,x_3,\ldots,x_{2^n+1}$, and maps each of $x_2,x_4,\ldots,x_{2^n}$ to itself times some product of $x_1,x_3,\ldots,x_{2^n+1}$. The next claim gives some restrictions on these odd-index factors:

\begin{mdframed}
Claim 7: Let $i,j\in\{1,2,\ldots,2^{n-1}\}$. Then the following hold:
\begin{enumerate}
\item $\alpha_{2i,2i-1}=0$.
\item $\alpha_{2j,2i-1}=\alpha_{2i,2j-1}$ and $\alpha_{2j,2i+1}=\alpha_{2i,2j+1}$.
\end{enumerate}
\end{mdframed}

Indeed, for statement (1), observe that if $\alpha_{2i,2i-1}=1$, then
\[
x_{2i}^{\alpha}\equiv x_{2i}\cdot x_1^{\alpha_{2i,1}}x_3^{\alpha_{2i,3}}\cdots x_{2i-3}^{\alpha_{2i,2i-3}}x_{2i-1}x_{2i+1}^{\alpha_{2i,2i+1}}\cdots x_{2^n+1}^{\alpha_{2i,2^n+1}}\Mod{C_n}
\]
and thus
\[
(x_{2i}^{\alpha})^2=[x_{2i-1},x_{2i}]\cdot[x_{2i},x_{2i+1}]^{\alpha_{2i,2i+1}}\cdot x_{2i}^2\cdot\prod_{k=1}^{2^{n-1}}{x_{2k+1}^{2\alpha_{2i,2k+1}}}=a\cdot b^{\alpha_{2i,2i+1}}\cdot b^{\alpha_{2i,1}+\alpha_{2i,2^n+1}}\not=1,
\]
which contradicts the fact that $\ord(x_{2i})=2$.

For statement (2), note that
\[
x_{2i}^{\alpha}\equiv x_{2i}\cdot x_1^{\alpha_{2i,1}}x_3^{\alpha_{2i,3}}\cdots x_{2i-3}^{\alpha_{2i,2i-3}}x_{2i+1}^{\alpha_{2i,2i+1}}x_{2i+3}^{\alpha_{2i,2i+3}}\cdots x_{2^n+1}^{\alpha_{2i,2^n+1}}\Mod{C_n}
\]
and
\[
x_{2j}^{\alpha}\equiv x_{2j}\cdot x_1^{\alpha_{2j,1}}x_3^{\alpha_{2j,3}}\cdots x_{2j-3}^{\alpha_{2j,2j-3}}x_{2j+1}^{\alpha_{2j,2j+1}}x_{2j+3}^{\alpha_{2j,2j+3}}\cdots x_{2^n+1}^{\alpha_{2j,2^n+1}}\Mod{C_n}.
\]
It follows from $1=[x_{2i},x_{2j}]$ that
\begin{align*}
1 &=[x_{2i-1},x_{2i}]^{\alpha_{2j,2i-1}}[x_{2i},x_{2i+1}]^{\alpha_{2j,2i+1}}[x_{2j-1},x_{2j}]^{\alpha_{2i,2j-1}}[x_{2j},x_{2j+1}]^{\alpha_{2i,2j+1}} \\
&=a^{\alpha_{2j,2i-1}+\alpha_{2i,2j-1}}\cdot b^{\alpha_{2j,2i+1}+\alpha_{2i,2j+1}},
\end{align*}
whence indeed, $\alpha_{2j,2i-1}=\alpha_{2i,2j-1}$ and $\alpha_{2j,2i+1}=\alpha_{2i,2j+1}$, as required.

\begin{mdframed}
Claim 8: The following hold:
\begin{enumerate}
\item For each $i\in\{1,2,\ldots,2^{n-1}-1\}$, the subgroup $\langle x_{2i},C_n\rangle$ is characteristic in $G_n$.
\item The coset union $x_{2^n}C_n \cup x_{2^n}x_{2^n+1}C_n$ is a characteristic subset of $G_n$.
\end{enumerate}
\end{mdframed}

Note that by Claim 7(2), if we know that $\alpha_{2k,2l-1}=0$ for some $k\in\{1,2,\ldots,2^{n-1}\}$ and some $l\in\{1,2,\ldots,2^{n-1}+1\}$, then we can actually conclude that $\alpha_{e,o}=0$ for all pairs $(e,o)\in\{1,2,\ldots,2^n+1\}^2$ where $e$ is even, $o$ is odd and $e+o=2k+2l-1$. This is because
\[
\alpha_{2k,2l-1}=\alpha_{2l,2k-1}=\alpha_{2k-2,2l+1}\text{ if }2k-2,2l+1\in\{1,\ldots,2^n+1\}
\]
and
\[
\alpha_{2k,2l-1}=\alpha_{2l-2,2k+1}=\alpha_{2k+2,2l-3}\text{ if }2l-3,2k+2\in\{1,2,\ldots,2^n+1\}.
\]
Therefore, by Claim 7(1), we conclude that $\alpha_{e,o}=0$ whenever $e+o\equiv3\Mod{4}$. We claim that more generally, for each $k=2,3,\ldots,n+1$, $\alpha_{e,o}=0$ whenever $e+o\equiv 1+2^{k-1}\Mod{2^k}$. We will show this by induction on $k$, with the induction base, $k=2$, done just above. So assume now that $k\leqslant n$, and that we know that $\alpha_{e,o}=0$ whenever $e+o\equiv 1+2^{k-1}\Mod{2^k}$. Then, in particular, $\alpha_{\epsilon,1}=\alpha_{\epsilon,2^n+1}=0$ whenever $\epsilon\in\{1,2\ldots,2^n+1\}$ and $\epsilon\equiv2^{k-1}\Mod{2^k}$. If $\alpha_{\epsilon,\epsilon+1}=1$, it follows (in view of Claims 6 and 7(1)) that
\[
(x_{\epsilon}^{\alpha})^2=[x_{\epsilon},x_{\epsilon+1}]=b\not=1,
\]
a contradiction. Hence $\alpha_{\epsilon,\epsilon+1}=0$ for all $\epsilon\in\{1,2,\ldots,2^n+1\}$ with $\epsilon\equiv 2^{k-1}\Mod{2^k}$, and thus $\alpha_{e,o}=0$ for all pairs $(e,o)\in\{1,2,\ldots,2^n+1\}^2$ where $e$ is even, $o$ is odd and $e+o\equiv 1+2^k\Mod{2^{k+1}}$, as we wanted to show.

An equivalent reformulation of what we just proved by induction on $k$ is that $\alpha_{e,o}=0$ for all pairs $(e,o)\in\{1,2,\ldots,2^n+1\}^2$ where $e$ is even and $o$ is odd \emph{unless} $e+o\equiv 1\Mod{2^{n+1}}$, i.e., unless $(e,o)=(2^n,2^n+1)$. Together with Claim 6, this proves Claim 8.

We can now conclude the proof of Proposition \ref{2GroupProp} as follows: By Claims 5 and 8, we have that modulo $\Aut_{\cent}(G_n)$, every automorphism of $G_n$ either fixes all generators of $G_n$, or it maps $x_{2^n}\mapsto x_{2^n}x_{2^n+1}$ while fixing all the other generators of $G_n$. In other words, modulo $\Aut_{\cent}(G_n)$, every automorphism of $G_n$ is equal to $\id_{G_n}$ or $\alpha_n$ as defined in Remark \ref{2GroupRem}(3), which is just what we wanted to show (see the beginning of this proof).
\end{proof}

\section{Finite groups \texorpdfstring{$G$}{G} with both \texorpdfstring{$\maol(G)$}{maol(G)} and \texorpdfstring{$d(G)$}{d(G)} bounded}\label{sec4}

This section is concerned with the proof of Theorem \ref{mainTheo}(3). Recall from Subsection \ref{subsecPrep1} that $d(G)$ denotes the minimum size of a generating subset of the finite group $G$. We note that if $G$ is any finite group with $\maol(G)\leqslant c$, then in particular, all conjugacy classes of $G$ are of length at most $c$, and so if $d(G)\leqslant d$, then the center $\zeta G$, being the intersection of the centralizers of the elements of any fixed generating subset of $G$, has index at most $c^d$ in $G$. Hence an upper bound on $|G|$ could be derived from an explicit version of Robinson and Wiegold's theorem \cite[Theorem 1]{RW84a}, more precisely from an explicit upper bound on $|\zeta G|$ for all finite groups $G$ with $\maol(G)\leqslant c$. As noted in \cite[Remark (i) at the end of Section 1]{RW84a}, the proof of \cite[Theorem 1]{RW84a} actually provides such an explicit upper bound, but it is complicated and was not worked out explicitly by Robinson and Wiegold.

Rather than proving our Theorem \ref{mainTheo}(3) by making Robinson and Wiegold's result explicit, we will exploit the fact that a related, celebrated result of B.H.~Neumann (which motivated Robinson and Wiegold's paper) has known explicit versions. This also means that modulo known, explicitly spelled out results, our proof will be elementary (the Robinson-Wiegold proof uses cohomological methods).

A \emph{BFC-group} is a group $G$ such that the maximum conjugacy class length in $G$ is bounded from above by some constant. The above mentioned theorem of B.H.~Neumann states that a group $G$ is a BFC-group if and only if the commutator subgroup $G'$ is finite (see \cite[Theorem 3.1]{Neu54a}). Later, an explicit version of Neumann's theorem was proved by Wiegold \cite[Theorem 4.7]{Wie57a}, stating that if $G$ is a group in which all conjugacy classes are of length at most $\ell$, then the order of $G'$ is at most $f(\ell)$ for some explicit function $f$. The currently best known choice for $f$ is the one from the following theorem:

\begin{theorem}\label{gmTheo}
(Guralnick-Mar{\'o}ti, \cite[Theorem 1.9]{GM11a})
Let $\ell$ be a positive integer, and let $G$ be a group such that all conjugacy classes of $G$ are of length at most $\ell$. Then
\[
|G'|\leqslant\ell^{\frac{1}{2}(7+\log{\ell})}.
\]
\end{theorem}

In our proof of Theorem \ref{mainTheo}(3), we will also need some simple lower bounds on the number of automorphisms of a finite abelian $p$-group $P$, which can be derived from the following exact formula for $|\Aut(P)|$:

\begin{theorem}\label{hrTheo}
(Hillar-Rhea, \cite[Theorem 4.1]{HR07a})
Let $p$ be a prime, and let $P$ be a finite abelian $p$-group. Write $P\cong\IZ/p^{e_1}\IZ\times\cdots\times\IZ/p^{e_n}\IZ$ with $1\leqslant e_1\leqslant\cdots\leqslant e_n$. For $k=1,\ldots,n$, set $d_k:=\max\{l\in\{1,\ldots,n\}\mid e_l=e_k\}$ and $c_k:=\min\{l\in\{1,\ldots,n\}\mid e_l=e_k\}$. Then
\[
|\Aut(P)|=\prod_{k=1}^n{(p^{d_k}-p^{k-1})}\cdot\prod_{j=1}^n{p^{e_j(n-d_j)}}\prod_{i=1}^n{p^{(e_i-1)(n-c_i+1)}}.
\]
\end{theorem}

\begin{corollary}\label{hrCor}
With notation as in Theorem \ref{hrTheo}, and assuming that $P$ is nontrivial, we have $|\Aut(P)|\geqslant\max\{p-1,p^{e_n-1}\}$.
\end{corollary}

\begin{proof}
Note that by definition, $d_1\geqslant 1$ and $c_n\leqslant n$. It follows that
\[
|\Aut(P)|=(p^{d_1}-1)\cdot\prod_{k=2}^n{(p^{d_k}-p^{k-1})}\cdot\prod_{j=1}^n{p^{e_j(n-d_j)}}\prod_{i=1}^n{p^{(e_i-1)(n-c_i+1)}}\geqslant p-1
\]
and
\begin{align*}
|\Aut(P)| &=\prod_{k=1}^n{(p^{d_k}-p^{k-1})}\cdot\prod_{j=1}^n{p^{e_j(n-d_j)}}\prod_{i=1}^{n-1}{p^{(e_i-1)(n-c_i+1)}}\cdot p^{(e_n-1)(n-c_n+1)} \\
&\geqslant1\cdot p^{(e_n-1)(n-c_n+1)}\geqslant p^{(e_n-1)(n-n+1)}=p^{e_n-1},
\end{align*}
as required.
\end{proof}

Furthermore, we will make use of the following upper bound on the first Chebyshev function:

\begin{theorem}\label{rsTheo}
(Rosser-Schoenfeld, \cite[Theorem 9]{RS62a})
Let
\[
\vartheta:\left[0,\infty\right)\rightarrow\left[0,\infty\right), x\mapsto\sum_{p\leqslant x}{\log{p}},
\]
where the summation index $p$ ranges over primes, be the first Chebyshev function. Then for all $x>0$, $\vartheta(x)<1.01624x$.
\end{theorem}

The following elementary upper bound on the number of automorphisms of a finite group will also be used:

\begin{lemma}\label{autBoundLem}
Let $G$ be a finite group. Then $|\Aut(G)|\leqslant|G|^{\log(|G|)/\log{2}}$.
\end{lemma}

\begin{proof}
Let $S=\{x_1,\ldots,x_{d(G)}\}$ be a (necessarily minimal) generating subset of $G$ of size $d(G)$. Then, setting $G_i:=\langle x_i,\ldots,x_{d(G)}\rangle$ for $i=1,\ldots d(G)+1$, we obtain a subgroup series $G=G_1>G_2>\cdots>G_{d(G)}>G_{d(G)+1}=\{1_G\}$. By Lagrange's theorem, $|G_{i+1}|\geqslant 2|G_i|$ for each $i\in\{1,\ldots,d(G)\}$, and so $|G|\geqslant 2^{d(G)}$, whence $|S|=d(G)\leqslant\frac{\log{|G|}}{\log{2}}$. The function which assigns to each automorphism of $G$ its restriction to $S$ is an injection, and so $|\Aut(G)|$ is at most the number of functions $S\rightarrow G$, which is exactly $|G|^{|S|}\leqslant|G|^{\log(|G|)/\log{2}}$.
\end{proof}

Finally, we will need generalizations of the concepts of a \enquote{standard tuple} and of the \enquote{power-commutator tuple} associated to a standard tuple as defined in the proof of Proposition \ref{maol2Prop}.

\begin{definition}\label{standardTupleDef}
Consider the following concepts:
\begin{enumerate}
\item Let $p$ be a prime, and let $P$ be a finite abelian $p$-group. Write $P\cong\IZ/p^{e_1}\IZ\times\cdots\times\IZ/p^{e_m}\IZ$ with $1\leqslant e_1\leqslant\cdots\leqslant e_m$. For $n\in\IN^+$ with $n\geqslant m$, a \emph{length $n$ standard generating tuple of $P$} is an $n$-tuple $(x_1,\ldots,x_n)\in P^n$ such that $P=\langle x_1,\ldots,x_n\rangle$ and for $i\in\{1,\ldots,n\}$,
\[
\ord(x_i)=\begin{cases}p^{e_i}, & \text{if }i\leqslant m, \\ 1, & \text{if }i>m.\end{cases}
\]
\item Let $H$ be a finite abelian group, say with $d(H)=n$. For $k\in\IN^+$, denote by $p_k$ the $k$-th prime, and by $P_k$ the Sylow $p_k$-subgroup of $H$. Hence up to isomorphism, we can write $H=\prod_{k\geqslant1}{P_k}$. A \emph{standard generating tuple of $H$} is an $n$-tuple $(h_1,\ldots,h_n)\in H^n$ such that $H=\langle h_1,\ldots,h_n\rangle$ and for each $k\geqslant1$, the entry-wise projection of $(h_1,\ldots,h_n)$ to $P_k$ is a length $n$ standard generating tuple of $P_k$.
\item Let $G$ be a finite group, and let $n:=d(G/G')$. A \emph{standard tuple in $G$} is an $n$-tuple $(g_1,\ldots,g_n)\in G^n$ whose entry-wise image under the canonical projection $G\rightarrow G/G'$ is a standard generating tuple of $G/G'$.
\end{enumerate}
\end{definition}

\begin{remark}\label{standardTupleRem}
Let $H$ be a finite abelian group, and let $G$ be an arbitrary finite group.
\begin{enumerate}
\item All standard generating tuples of $H$ are polycyclic generating sequences of $H$, and they all induce the same power-commutator presentation of $H$. Moreover, any polycylic generating sequence of $H$ inducing this said power-commutator presentation is a standard generating tuple. Hence $\Aut(H)$ acts $1$-transitively on the standard generating tuples of $H$, and so the number of standard generating tuples of $H$ is exactly $|\Aut(H)|$.
\item The number of standard tuples in $G$ is exactly $|\Aut(G/G')|\cdot|G'|^{d(G/G')}$.
\item For each standard tuple $(g_1,\ldots,g_n)$ in $G$, we have $G=\langle g_1,\ldots,g_n,G'\rangle$.
\end{enumerate}
\end{remark}

\begin{definition}\label{powerCommDef}
Let $G$ be a finite group, let $n:=d(G/G')$, and let $(g_1,\ldots,g_n)$ be a standard tuple in $G$.
\begin{enumerate}
\item The \emph{power-automorphism-commutator tuple associated with $(g_1,\ldots,g_n)$} is the $(2n+{n\choose 2})$-tuple
\[
(\pi_1,\ldots,\pi_n,\alpha_1,\ldots,\alpha_n,\gamma_{1,1},\gamma_{1,2},\ldots,\gamma_{1,n},\gamma_{2,3},\gamma_{2,4},\ldots,\gamma_{2,n},\ldots,\gamma_{n-1,n})
\]
with entries in $G'\cup\Aut(G')$ such that
\begin{itemize}
\item $\pi_i=g_i^{\ord_{G/G'}(g_iG')}\in G'$ for $i=1,\ldots,n$,
\item $\alpha_i\in\Aut(G')$ is the automorphism induced through conjugation by $g_i$ for $i=1,\ldots,n$, and
\item $\gamma_{i,j}=[g_i,g_j]\in G'$ for $1\leqslant i<j\leqslant n$.
\end{itemize}
\item Two standard tuples in $G$ are called \emph{equivalent} if and only if they have the same associated power-automorphism-commutator tuple.
\end{enumerate}
\end{definition}

\begin{remark}\label{powerCommRem}
Let $G$ be a finite group, let $H:=G/G'$, and let $n:=d(G/G')$. Every standard generating tuple $(h_1,\ldots,h_n)$ of $H$ is a polycyclic generating sequence of $H$, with respect to which $H$ has the power-commutator presentation
\[
H=\langle x_1,\ldots,x_n\mid x_i^{\ord(h_i)}=1\text{ for }i=1,\ldots,n; [x_i,x_j]=1\text{ for }1\leqslant i<j\leqslant n\rangle.
\]
Now, let $(c_1,\ldots,c_m)$ be a fixed generating tuple of $G'$, with respect to which $G'$ has the presentation
\[
G'=\langle y_1,\ldots,y_m\mid \rho_j=1\text{ for }j=1,\ldots,k\rangle
\]
with $\rho_j$ an element of the free group $\F(y_1,\ldots,y_m)$ for $j=1,\ldots,k$. Then with respect to any (generating) $(m+n)$-tuple of the form $(g_1,\ldots,g_n,c_1,\ldots,c_m)\in G^{m+n}$, where $(g_1,\ldots,g_n)$ is a standard tuple in $G$, the group $G$ has a presentation of the form
\begin{align*}
G=\langle x_1,\ldots,x_n,y_1,\ldots,y_m\mid &\rho_j=1\text{ for }j=1,\ldots,m; x_i^{o_i}=w_i\text{ for }i=1,\ldots,n; \\
& [x_i,x_j]=w_{i,j}\text{ for }1\leqslant i<j\leqslant n; \\
&y_k^{x_i}=v_{i,k}\text{ for }i=1,\ldots,n\text{ and }k=1,\ldots,m\rangle,
\end{align*}
where $o_i$ denotes the common order of the $i$-th entry of any standard generating tuple of $H=G/G'$, and $w_i,w_{i,j},v_{i,k}\in\F(y_1,\ldots,y_m)$.

From this, it is clear that any two equivalent standard tuples in $G$ lie in the same orbit of the component-wise action of $\Aut(G)$; in fact, they are conjugate under an automorphism of $G$ which fixes $G'$ element-wise.
\end{remark}

We are now ready for the

\begin{proof}[Proof of Theorem \ref{mainTheo}(3)]
Let $G$ be a finite group with $\maol(G)\leqslant c$ and $d(G)\leqslant d$. Then in particular, all conjugacy classes of $G$ are of length at most $c$. It follows from Theorem \ref{gmTheo} that
\[
|G'|\leqslant c^{\frac{1}{2}(7+\log{c})}.
\]
Our goal will thus be to bound $|G:G'|$ explicitly from above in terms of $c$ and $d$. First, we show the following

\begin{mdframed}
Claim: If $p$ is a prime divisor of $|G:G'|$, then
\[
p\leqslant A(c,d)+1=c^{d+\frac{1}{2}(7+\log{c})\left({d \choose 2}+\frac{d}{2\log{2}}\cdot(7+\log{c})\log{c}\right)}+1.
\]
\end{mdframed}

In order to prove the Claim, observe that by Corollary \ref{hrCor} and Remark \ref{standardTupleRem}(2), the number of standard tuples in $G$ is at least
\[
(p-1)\cdot|G'|^{d(G/G')}.
\]
On the other hand, in view of Lemma \ref{autBoundLem}, the number of equivalence classes of standard tuples in $G$ is at most
\[
|G'|^{d(G/G')+d(G/G')\frac{\log{|G'|}}{\log{2}}+{d(G/G') \choose 2}}.
\]
It follows that there is an equivalence class of standard tuples in $G$ which is of size at least
\[
\frac{(p-1)|G'|^{d(G/G')}}{|G'|^{d(G/G')+d(G/G')\frac{\log{|G'|}}{\log{2}}+{d(G/G') \choose 2}}}=\frac{p-1}{|G'|^{d(G/G')\frac{\log{|G'|}}{\log{2}}+{d(G/G') \choose 2}}}.
\]
On the other hand, since $c\geqslant\maol(G)$, all $\Aut(G)$-orbits on $d(G/G')$-tuples over $G$ are of length at most $c^{d(G/G')}$. In view of Remark \ref{powerCommRem}, it follows that
\[
c^{d(G/G')}\geqslant\frac{p-1}{|G'|^{d(G/G')\frac{\log{|G'|}}{\log{2}}+{d(G/G') \choose 2}}},
\]
and hence
\[
p \leqslant c^{d(G/G')}\cdot|G'|^{d(G/G')\frac{\log{|G'|}}{\log{2}}+{d(G/G') \choose 2}}+1 \leqslant c^d\cdot c^{\frac{1}{2}(7+\log{c})(d\frac{\frac{1}{2}(7+\log{c})\log{c}}{\log{2}}+{d\choose 2})}+1,
\]
as asserted by the Claim.

Now that the Claim has been proved, let $f$ denote the largest exponent $e$ occurring in the (essentially unique) direct factor decomposition of $G/G'$ into primary cylic groups $\IZ/p^e\IZ$. Then by the above Claim and the fact that $G/G'$ is $d$-generated, we have (letting the variable $p$ range over primes)
\[
\frac{|G|}{c^{\frac{1}{2}(7+\log{c})}}\leqslant|G:G'|\leqslant\prod_{p\leqslant A(c,d)+1}{p^{df}}=\exp(\vartheta(A(c,d)+1)\cdot df),
\]
and thus, in view of Theorem \ref{rsTheo},
\begin{equation}\label{astEq}
f \geqslant \frac{\log{|G|}-\frac{1}{2}(7+\log{c})\log{c}}{d\cdot\vartheta(A(c,d)+1)} \geqslant \frac{\log{|G|}-\frac{1}{2}(7+\log{c})\log{c}}{1.01624d\cdot(A(c,d)+1)} =: g(|G|,c,d).
\end{equation}
By Corollary \ref{hrCor} and Remark \ref{standardTupleRem}(2), the number of standard tuples in $G$ is at least
\[
2^{f-1}\cdot|G'|^{d(G/G')} \geqslant 2^{g(|G|,c,d)-1}\cdot|G'|^{d(G/G')}.
\]
On the other hand, the number of equivalence classes of standard tuples in $G$ is at most
\[
|G'|^{d(G/G')+d(G/G')\frac{\log{|G'|}}{\log{2}}+{d(G/G') \choose 2}}.
\]
It follows that there is an equivalence class of standard tuples in $G$ which is of size at least
\[
\frac{2^{g(|G|,c,d)-1}\cdot|G'|^{d(G/G')}}{|G'|^{d(G/G')+d(G/G')\frac{\log{|G'|}}{\log{2}}+{d(G/G') \choose 2}}}=\frac{2^{g(|G|,c,d)-1}}{|G'|^{d(G/G')\frac{\log{|G'|}}{\log{2}}+{d(G/G') \choose 2}}}.
\]
But again, since $\maol(G)\leqslant c$, the length of an $\Aut(G)$-orbit on $d(G/G')$-tuples over $G$ cannot exceed $c^{d(G/G')}$, and so, in view of Remark \ref{powerCommRem},
\[
c^{d(G/G')}\geqslant\frac{2^{g(|G|,c,d)-1}}{|G'|^{d(G/G')\frac{\log{|G'|}}{\log{2}}+{d(G/G') \choose 2}}},
\]
which implies that
\begin{align*}
2^{g(|G|,c,d)-1} &\leqslant c^{d(G/G')}\cdot|G'|^{d(G/G')\frac{\log{|G'|}}{\log{2}}+{d(G/G') \choose 2}}\leqslant c^d\cdot c^{\frac{1}{2}(7+\log{c})\left(d\frac{\frac{1}{2}(7+\log{c})\log{c}}{\log{2}}+{d \choose 2}\right)} \\
&=A(c,d).
\end{align*}
It follows that
\[
g(|G|,c,d)\leqslant \frac{\log{A(c,d)}}{\log{2}}+1,
\]
or, equivalently (in view of the definition of $g(|G|,c,d)$ in Formula (\ref{astEq}) above),
\[
\log{|G|}\leqslant1.01624d\cdot\left(A(c,d)+1\right)\cdot\left(\frac{\log{A(c,d)}}{\log{2}}+1\right)+\frac{1}{2}(7+\log{c})\log{c},
\]
which is just what we needed to show.
\end{proof}

\section{Finite groups \texorpdfstring{$G$}{G} with \texorpdfstring{$\maol(G)\leqslant23$}{maol(G)<=23}}\label{sec5}

This section is concerned with the proof of Theorem \ref{mainTheo}(4). Let us first introduce a shorthand notation for a concept that was already implicit in the previous section:

\begin{notation}\label{mcclNot}
Let $G$ be a finite group. We denote by
\[
\mccl(G):=\max_{g\in G}{|g^G|}
\]
the maximum conjugacy class length of $G$.
\end{notation}

The following lemma will prove useful in our proof of Theorem \ref{mainTheo}(4):

\begin{lemma}\label{crLem}
Let $T$ be a finite group that can be written as a nonempty direct product of nonabelian finite simple groups. Assume that $\mccl(T)\leqslant23$. Then $T\cong\Alt(5)$.
\end{lemma}

\begin{proof}
We first show the following, weaker

\begin{mdframed}
Claim: Let $S$ be a nonabelian finite simple group with $\mccl(S)\leqslant23$. Then $S\cong\Alt(5)$.
\end{mdframed}

Using the ATLAS of Finite Groups \cite{CCNPW85a}, one can check that $\mccl(S)>23$ for all sporadic finite simple groups $S$. Moreover, if $S=\Alt(m)$ with $m\geqslant6$, then the length of the $S$-conjugacy class of any $3$-cycle in $S$ is
\[
2\cdot{m \choose 3}=\frac{m(m-1)(m-2)}{3}\geqslant\frac{6\cdot 5\cdot 4}{3}=40>23.
\]
It remains to show that if $S$ is a nonabelian finite simple group of Lie type with $\mccl(S)\leqslant23$, then $S\cong A_1(4)\cong A_1(5)\cong\Alt(5)$. To that end, note that if $\mccl(S)\leqslant23$, then $S$ has a proper subgroup (namely an element centralizer) of index at most $23$, and so $m(S)\leqslant23$, where $m(S)$ denotes the minimum faithful permutation representation degree of $S$ (or, equivalently, the smallest index of a maximal subgroup of $S$). The values of $m(S)$ when $S$ is a finite simple group of Lie type can be found in \cite[Table 4, p.~7682]{GMPS15a}, and using this information, it is easy to check that $m(S)\leqslant23$ unless $S\cong A_d(q)\cong\PSL_{d+1}(q)$ with $(d,q)$ from the set
\[
\{(1,5),(1,7),(1,8),(1,9),(1,11),(1,13),(1,16),(1,17),(1,19),(2,2),(2,3),(2,4),(3,2)\}.
\]
By going through this finite list of groups $S$ and computing $\mccl(S)$ with GAP \cite{GAP4}, one finds that indeed, $S=A_1(5)\cong\Alt(5)$ is the only nonabelian finite simple group with $\mccl(S)\leqslant23$.

Now that the Claim is proved, we can conclude as follows: Write $T=S_1^{n_1}\times\cdots\times S_r^{n_r}$ where $S_1,\ldots,S_r$ are pairwise nonisomorphic nonabelian finite simple groups and $n_1,\ldots,n_r\in\IN^+$. Then, since the conjugacy classes of a direct product $G_1\times G_2$ are just the Cartestian products of the conjugacy classes of $G_1$ with the conjugacy classes of $G_2$, we find that
\[
23\geqslant\mccl(T)=\prod_{i=1}^r{\mccl(S_i)^{n_i}}.
\]
Hence, by the above Claim, we have $r=1$ and $S_1=\Alt(5)$, so $T\cong\Alt(5)^n$ for some $n\in\IN^+$. But if $n\geqslant 2$, then
\[
23\geqslant\mccl(T)=\mccl(\Alt(5))^n=20^n\geqslant 20^2=400>23,
\]
a contradiction. Therefore, $T\cong\Alt(5)$, as we needed to show.
\end{proof}

We are now ready for the

\begin{proof}[Proof of Theorem \ref{mainTheo}(4)]
We proceed by contradiction: Assume that $G$ is a finite nonsolvable group with $\maol(G)\leqslant23$. Recall the facts on finite semisimple groups listed in Subsection \ref{subsecPrep4}. We have that $G/\Rad(G)$ is a nontrivial finite semisimple group, and
\[
23\geqslant\maol(G)\geqslant\mccl(G)\geqslant\mccl(G/\Rad(G))\geqslant\mccl(\Soc(G/\Rad(G))).
\]
Since $\Soc(G/\Rad(G))$ is a nonempty direct product of nonabelian finite simple groups, Lemma \ref{crLem} yields that $\Soc(G/\Rad(G))\cong\Alt(5)$, and thus $G/\Rad(G)$ is isomorphic to either $\Alt(5)$ or $\Sym(5)$.  However,
\[
\mccl(\Sym(5))=24>23,
\]
so we conclude that $G/\Rad(G)\cong\Alt(5)$. We now show the following

\begin{mdframed}
Claim 1: Let $x\in G/\Rad(G)$, and let $\tilde{x}$ be a lift of $x$ in $G$. Then the conjugacy class $\tilde{x}^G$ consists of exactly one element from each of the cosets of $\Rad(G)$ which correspond to the elements of the conjugacy class $x^{G/\Rad(G)}$. In particular, $\Rad(G)=\zeta G$.
\end{mdframed}

For the proof of Claim 1, assume first that $x$ is a \emph{nontrivial} element of $G/\Rad(G)$. Then since $G/\Rad(G)\cong\Alt(5)$, we have
\[
|x^{G/\Rad(G)}|\geqslant 12>\frac{23}{2}.
\]
Hence the conjugacy class length $|\tilde{x}^G|$, being a multiple of $|x^{G/\Rad(G)}|$, must be equal to $|x^{G/\Rad(G)}|$, and the assertion follows for $x$. As for $x=1_{G/\Rad(G)}$, the assertion is equivalent to $\Rad(G)\leqslant\zeta G$, which we can prove as follows. Fix a nontrivial element $y\in G/\Rad(G)$, let $\tilde{y}$ be a lift of $y$ in $G$, and let $r\in\Rad(G)$ be arbitrary. Then
\[
(\tilde{y}r)^{\tilde{y}}=\tilde{y}r^{\tilde{y}}\in\tilde{y}\Rad(G)\cap(\tilde{y}r)^G=(\tilde{y}r)\Rad(G)\cap(\tilde{y}r)^G=\{\tilde{y}r\},
\]
whence $r^{\tilde{y}}=r$. This shows that $\C_G(r)$ contains all of $G\setminus\Rad(G)$, and thus $\C_G(r)=G$, i.e., $r\in\zeta G$. This concludes the proof of the main assertion, which involved showing that $\Rad(G)\leqslant\zeta G$. As for the \enquote{In particular}, namely that $\Rad(G)=\zeta G$, just use that $G/\Rad(G)\cong\Alt(5)$ is centerless.

Claim 1 implies the following

\begin{mdframed}
Claim 2: Let $g_1,g_2\in G$. Then $g_1$ and $g_2$ commute if and only if their images in $G/\Rad(G)$ commute.
\end{mdframed}

Note that the implication \enquote{$\Rightarrow$} in Claim 2 is trivial, so we focus on proving the implication \enquote{$\Leftarrow$}. Let $x_1$ and $x_2$ be commuting elements of $G/\Rad(G)$, and let $\widetilde{x_1}$ and $\widetilde{x_2}$ be lifts in $G$ of $x_1$ and $x_2$ respectively. We need to show that $\widetilde{x_1}$ and $\widetilde{x_2}$ commute. Since $x_1$ and $x_2$ commute in $G/\Rad(G)$, we conclude that
\[
\widetilde{x_1}^{\widetilde{x_2}} \in \widetilde{x_1}\Rad(G)\cap\widetilde{x_1}^G=\{\widetilde{x_1}\},
\]
where the equality is by Claim 1. Hence $\widetilde{x_1}^{\widetilde{x_2}}=\widetilde{x_1}$, which just means that $\widetilde{x_1}$ and $\widetilde{x_2}$ commute, as we wanted to show.

By Claim 2 and the facts that $\Rad(G)=\zeta G$ (see Claim 1) and that the Sylow subgroups of $G/\Rad(G)\cong\Alt(5)$ are abelian, it follows that the Sylow subgroups of $G$ are abelian. Hence, by \cite[result 10.1.7, p.~289]{Rob96a}, we have
\[
G'\cap\Rad(G)=G'\cap\zeta G=\{1_G\}.
\]
But since $G/\Rad(G)\cong\Alt(5)$ is perfect, $G=\langle\Rad(G),G'\rangle$, whence
\[
G=\Rad(G)\times G'.
\]
It follows that
\[
G'\cong G/\Rad(G)\cong\Alt(5),
\]
and by Lemma \ref{simpleLem}(1), we find that
\[
23\geqslant\maol(G)\geqslant\maol(G')=\maol(\Alt(5))=24,
\]
a contradiction.
\end{proof}

\section{Concluding remarks}\label{sec6}

We conclude this paper with some related open problems for further research. Arguably the most glaring open problem, arising when comparing statements (1) and (2) of Theorem \ref{mainTheo}(1), is the following:

\begin{problem}\label{open1}
Detemine the largest positive integer $c_0$ such that there are only finitely many finite groups $G$ with $\maol(G)\leqslant c_0$ (and, if possible, list those finitely many $G$).
\end{problem}

Observe that by Theorem \ref{mainTheo}(1,2), we have $c_0\in\{3,4,5,6,7\}$. The next problem is motivated by the fact that the $2$-groups discussed in Section \ref{sec3} \enquote{just} fail to have the property that all their automorphisms are central:

\begin{question}\label{open2}
Do there exist infinitely many finite groups $G$ with $|\zeta G|=4$ such that all automorphisms of $G$ are central?
\end{question}

If the answer to Question \ref{open2} is \enquote{yes}, then by Theorem \ref{mainTheo}(1), the constant $c_0$ from Problem \ref{open1} is $3$, and Problem \ref{open1} is solved completely by Theorem \ref{mainTheo}(1).

Finally, we would like to pose the following related problem on permutation groups:

\begin{problem}\label{open3}
Let $G\leqslant\Sym(\Omega)$ be a permutation group of finite degree, and set
\[
\maol_{\perm}(G):=\max_{g\in G}{|g^{\N_{\Sym(\Omega)}(G)}|}.
\]
Determine the largest non-negative integer $c_1$ such that all finite-degree permutation groups $G$ with $\maol_{\perm}(G)\leqslant c_1$ have constantly bounded order, and, if possible, classify those $G$. Is $c_1=c_0$, with $c_0$ as in Problem \ref{open1}?
\end{problem}

\end{document}